\documentclass[reqno,twoside,12pt]{amsart}
\usepackage{amssymb,amsfonts,amsthm,amsmath}
\usepackage{enumitem}
\usepackage[pagebackref=false]{hyperref}
\usepackage[hmargin=1in,vmargin=1in]{geometry}
\usepackage{cite}

\def\eqdef{\stackrel{\rm def}{=}}

\def\d{{\rm d}}
\def\ddt{\frac{\d}{\d t}}

\newcommand{\tnum}{\rm(\roman*)}
\newcommand{\rnum}{\rm(\alph*)}

\def\beq{\begin{equation}}
\def\eeq{\end{equation}}
\def\beqs{\begin{equation*}}
\def\eeqs{\end{equation*}}

\newtheorem{theorem}{Theorem}[section]
\newtheorem{lemma}[theorem]{Lemma}

\newtheorem{corollary}[theorem]{Corollary}
\newtheorem{definition}[theorem]{Definition}
\newtheorem{assumption}[theorem]{Assumption}

\theoremstyle{definition}
\newtheorem{remark}[theorem]{Remark}
\newtheorem{example}[theorem]{Example}
\newtheorem*{xnotation}{Notation}

\usepackage{color}

\def\varep{\varepsilon}

\newcommand{\R}{\ensuremath{\mathbb R}}

\newcommand{\N}{\ensuremath{\mathbb N}}

\newcommand{\bigo}{\mathcal O}

\numberwithin{equation}{section}

\title{\textbf{Long-time behavior of solutions of superlinear systems of differential equations}}
\author{Luan Hoang}

\address{Department of Mathematics and Statistics,
Texas Tech University\\
1108 Memorial Circle, Lubbock, TX 79409--1042, U. S. A.}
\email{luan.hoang@ttu.edu}

\makeatletter
\@namedef{subjclassname@2020}{
  \textup{2020} Mathematics Subject Classification}
\makeatother

\date{\today}

\begin{document}

\begin{abstract} This paper establishes the precise asymptotic behavior, as time $t$ tends to infinity, for nontrivial, decaying solutions of genuinely nonlinear systems of ordinary differential equations. The lowest order term in these systems, when the spatial variables are small, is not linear, but rather positively homogeneous of a degree greater than one. We prove that the solution behaves like $\xi t^{-p}$, as $t\to\infty$, for a nonzero vector $\xi$ and an explicit number $p>0$.
\end{abstract}

\maketitle


\tableofcontents

\pagestyle{myheadings}\markboth{L. Hoang}
{Long-time Behavior  of Solutions of Superlinear Systems of Differential Equations}

 
\section{Introduction}\label{intro}

This paper answers a fundamental question about the precise asymptotic behavior of decaying solutions, as time tends to infinity, of a genuine nonlinear system of ordinary differential equations. The word ``genuine'' is used roughly here to refer to the case when the main dissipation term in the system is nonlinear. This situation is different from many previously studied  systems of ordinary differential equations (ODE) and partial differential equations (PDE).
Our starting point is the following result by Foias and Saut \cite{FS84a} for solutions of the Navier--Stokes equations (NSE) in bounded or periodic domains. They consider the NSE written in the functional form in an appropriate functional space as
\beq\label{NSE}
u'+Au+B(u,u)=0,
\eeq
where $A$ is a linear operator with positive eigenvalues and $B(\cdot,\cdot)$ is a bilinear form.

For any nontrivial solution  $u(t)$ of \eqref{NSE}, there exist an eigenvalue $\Lambda$ of $A$ and an eigenfunction $\xi$ of $A$ associated with $\Lambda$ such that
\beq\label{FSlim}
e^{\Lambda t}u(t)\to \xi\text{ as $t\to\infty$.}
\eeq
Above, the limit holds in any $C^m$-norms.

This result is extended to more general differential inequalities in \cite{Ghidaglia1986a}. As a consequence of \cite{Ghidaglia1986a}, if $u(t)$ is a nontrivial, decaying solution of 
\beq\label{uAF}
u'+Au=F(u),
\eeq 
where $F(u)$ is a higher order term, as $u\to 0$, then the asymptotic approximation \eqref{FSlim} holds.
The proof in \cite{Ghidaglia1986a} follows that of Foias and Saut in \cite{FS84a}.
For systems of ODE, this result is re-established in \cite{CaHK1} using a different method.

Note that Foias and Saut also obtain the asymptotic expansions for the solutions of the NSE \eqref{NSE} in \cite{FS87}.
However, these asymptotic expansions can be obtained independently from the first approximation in \eqref{FSlim}.  
The interested reader can find other results on the asymptotic expansions for different ODE and PDE without forcing functions in \cite{Minea,Shi2000,HM1,HTi1,CaHK1}, and with forcing functions in \cite{HM2,CaH1,CaH2,CaH3,H5,H6}.
In particular,  \cite{CaHK1} obtains the asymptotic expansions of nontrivial, decaying solutions of \eqref{uAF} in the case $F(u)$ is not smooth in any neighborhood of the origin. Roughly speaking, $F(u)$ can be approximated, near the origin,  by a finite sum or a series of positively homogeneous terms of possibly nonintegral degrees, see Definition \ref{phom} below. A crucial step in  \cite{CaHK1} is the very first asymptotic approximation \eqref{FSlim}. 

All of the mentioned papers deal with the equations when the main dissipation term is linear.
Having motivated by  the general class of  positively homogeneous functions in \cite{CaHK1}, we study a quite different class of equations when the lowest order term is nonlinear. More specifically, we consider the following ODE system in $\R^n$
\beq\label{yReq}
y'=-H(y) Ay +G(t,y),
\eeq
where $A$ is a constant $n\times n$ matrix with positive eigenvalues, $H$ is a positively homogeneous function of degree $\alpha>0$, and $G$ is a higher order term. More precise conditions will be stated in Sections \ref{aesec}, \ref{symcase}, \ref{gencase} below. 

We study the nontrivial, decaying solutions $y(t)$ of \eqref{yReq} as $t\to\infty$.
Because of the lack of the linear term, $y(t)$ will not  decay exponentially. Rather, it will decay as a power function, see Theorem \ref{smallthm} below.
We present here a heuristic argument  in order to find out certain information about a possible asymptotic approximation for $y(t)$. 

We assume, as $t\to\infty$,  that 
\beq \label{yxip}
y(t)\sim\xi t^{-p}\text{ for some nonzero vector $\xi\in \R^n$ and real number $p>0$.}
\eeq 

Using this to approximate both sides of equation \eqref{yReq} and ignoring the higher order term $G(t,y)$, we obtain 
\beqs
-p t^{-p-1} \xi  \sim -t^{-p(\alpha+1)}H(\xi) A\xi.
\eeqs

Matching the power of $t$ and the coefficients from both sides gives 
\beq\label{heu0}
p=1/\alpha\text{ and }
A\xi=\frac{1}{\alpha H(\xi)}\xi.
\eeq

Thus, $\Lambda \eqdef 1/ (\alpha H(\xi))$ is an eigenvalue of $A$, and $\xi$ is an eigenvector of $A$ associated with $\Lambda$. It turns out that the asymptotic approximation \eqref{yxip} with \eqref{heu0} is exactly what we will obtain rigorously under some appropriate conditions on $A$, $H$ and $G$. This is the goal of the current paper.

The paper is organized as follows. 
In Section \ref{simeg}, we establish in Theorem \ref{simthm} the asymptotic behavior of the solutions to a simple system which comes up in many applications. 
In addition to its own merit, the result serves as a key step in the proof of the main results for more general cases in Sections \ref{symcase} and \ref{gencase}.
In Section \ref{aesec}, we specify the conditions for $A$, $H$ and $G$.
The basic issues of global existence, uniqueness and asymptotic estimates are established in Theorem \ref{smallthm}. The estimates in \eqref{yplus} actually justify the correct decaying mode that we attempted in the heuristic argument \eqref{yxip}. 
Section \ref{symcase} treats the case when the matrix $A$ is symmetric. The main result in this section is Theorem \ref{thmsym}, which basically proves \eqref{yxip} and \eqref{heu0}. 
Preparations for it consist of many steps from Lemma \ref{nzprop} to Lemma \ref{lem4}. 
The proof of Theorem \ref{thmsym} inherits  some original  ideas from the work of Foias and Saut \cite{FS84a}. They include studying the asymptotic behavior of the ``Dirichlet" quotient $\lambda(t)$ and the function $v(t)$ -- a normalization of the solution $y(t)$, see \eqref{quot}. 
However, to treat the genuine nonlinearity in the problem, new ideas and techniques are needed.
Notable among them are a new perturbation argument and the newly realized property (HC) for the function $H$, see Definition \ref{HCdef} and Assumption \ref{HHsphere}. 
Regarding the former, equation \eqref{yReq} is considered as a perturbation of the reduced equation \eqref{Rygood}, which is in the form of the basic case \eqref{basicf}. 
It is based on the projection of equation \eqref{yReq} to an appropriate eigenspace of $A$ and by freezing $v(t)$ to its limit $v_*$. The corresponding eigenvalue $\Lambda$ of this eigenspace turns out to be determined by the asymptotic behavior of the quotient $\lambda(t)$, see  Lemma \ref{lem1}.
This perturbation approach is facilitated by the function $H$ having the property (HC). Although it is a H\"older-like continuity condition, it is actually much weaker than the standard H\"older continuity. 
The most general result of the paper is Theorem \ref{mainthm} in Section \ref{gencase}. It is derived from Theorem \ref{thmsym} by the standard linear transformation using the equivalence of $A$ to a diagonal matrix in \eqref{Adiag}. 
Theorem \ref{xirel} shows that whenever the asymptotic approximation \eqref{yxip} is established, it dictates property \eqref{heu0} of the power $p$ and the vector $\xi$.
Many examples of the function $H$ are given in Example \ref{Heg}.
Finally, we briefly comment on the literature in Remark \ref{other}.

\begin{xnotation}
For any vector $x\in\R^n$, we denote by $|x|$ its Euclidean norm.

For an $n\times n$ real matrix $A=(a_{ij})_{1\le i,j\le n}$, its Euclidean norm is 
$\|A\|=(\sum_{i,j=1}^n a_{ij}^2)^{1/2}$, while its norm as a bounded linear operator is  $\|A\|_{\rm op}=\max\{|Ax|:|x|=1\}$.

For two functions $f,g:[T,\infty)\to [0,\infty)$, for some number $T\in\R$,  we write 
\beqs
f(t)=\bigo(g(t))\text{ as $t\to\infty$,}
\eeqs 
if there are $T'\ge T$  and $C>0$ such that $f(t)\le Cg(t)$ for all $t\ge T'$. Very often, ``as $t\to\infty$'' is implicitly understood and, hence, omitted.

Hereafter, $n\in\N=\{1,2,3,\ldots\}$ is the spatial dimension and is fixed.
\end{xnotation}

\section{A basic case}\label{simeg}
Let $a>0$, $\alpha>0$ and $t_*\ge 0$ be given numbers.  

Assume $y\in C^1([t_*,\infty),\R^n)$ satisfies $y(t)\ne 0$ for all $t\ge t_*$, 
\beq \label{limzero}
\lim_{t\to\infty} y(t)= 0,
\eeq 
and
\beq\label{basicf}
y'=-a|y|^\alpha y +f(t) \text{ for } t>t_*,
\eeq
where  $f$ is a continuous function from $[t_*,\infty)$ to $\R^n$ such that
\beq \label{frate}
|f(t)|\le M|y(t)|^{\alpha+1+\delta},\text{ for all $t\ge t_*$ and some constants $M,\delta>0$}.
\eeq 

The asymptotic behavior of $y(t)$, as $t\to\infty$, is the following.

\begin{theorem}\label{simthm}
There exists a nonzero vector $\xi_*\in\R^n$ such that, as $t\to\infty$,
\beq\label{yxi}
|y(t)- \xi_* t^{-1/\alpha}|=\bigo(t^{-1/\alpha-\varep})\text{ for some $\varep>0$.}
\eeq

Moreover,
\beq\label{axi1}
\alpha a |\xi_*|^\alpha=1.
\eeq
\end{theorem}
\begin{proof}
By increasing the initial time $t_*$ to be sufficiently large, we can assume, without loss of generality, that $|y(t)|\le 1$ for all $t\ge t_*$, and $\delta<\alpha$. 

For $t> t_*$, we calculate
\beq\label{dy3}
\ddt(|y|^{-\alpha})=-\frac{\alpha}{|y|^{\alpha+2}} y'\cdot y
=a\alpha - \frac{\alpha f(t)\cdot y}{|y|^{\alpha+2}}.
\eeq

By Cauchy-Schwarz's inequality and condition \eqref{frate}, we have 
\beq\label{fyyrate}
\frac{|f(t)\cdot y(t)|}{|y(t)|^{\alpha+2}}\le M|y(t)|^\delta \text{ for $t\ge t_*$.}
\eeq

We fix two positive numbers $a_1$ and $a_2$ such that $a_1<a<a_2$. 
Thanks to \eqref{limzero} and \eqref{fyyrate}, there is a positive number $t_0\ge t_*$ such that one has 
\beqs
\alpha a_1 \le \ddt|y|^{-\alpha}\le \alpha a_2 \text{ for all $t>t_0$.}
\eeqs 

Hence, for $t\ge t_0>0$,
\begin{align*}
&|y(t)|^\alpha\ge \frac{1}{|y(t_0)|^{-\alpha}+ \alpha a_2 (t-t_0)}
 \ge C_1^\alpha t^{-1},  \\
 &|y(t)|^\alpha \le \frac{1}{|y(t_0)|^{-\alpha}+ \alpha a_1 (t-t_0)}
\le C_2^\alpha t^{-1},
\end{align*}
where $C_1$ and $C_2$ are some positive numbers.
We obtain
\beq\label{roughy}
C_1 t^{-1/\alpha} \le |y(t)|\le C_2 t^{-1/\alpha}\text{ for all } t\ge t_0.
\eeq

As a consequence of \eqref{frate} and \eqref{roughy}, we have
\beq\label{fta}
|f(t)| \le M_1 t^{-1-1/\alpha-\delta/\alpha} \text{ for all $t\ge t_0$, where $M_1=M C_2^{1+\alpha+\delta}>0$.
}
\eeq

Integrating equation \eqref{dy3} gives
\beqs
|y(t)|^{-\alpha}-|y(t_0)|^{-\alpha}=a\alpha (t-t_0)+g(t), 
\text{ where } g(t)=-\alpha \int_{t_0}^t \frac{f(\tau)\cdot y(\tau)}{|y(\tau)|^{\alpha+2}}\d\tau.
\eeqs

Hence, for all $t\ge t_0$, one has $|y(t_0)|^{-\alpha}+a\alpha (t-t_0)+g(t)>0$ and 
\beq\label{ynorm}
|y(t)|^\alpha=\frac{1}{|y(t_0)|^{-\alpha}+a\alpha (t-t_0)+g(t)}.
\eeq

Using \eqref{fyyrate} and the upper bound of $|y(t)|$ in \eqref{roughy}, we estimate 
\begin{align*}
|g(t)|&\le \alpha \int_{t_0}^t \frac{|f(\tau)\cdot y(\tau)|}{|y(\tau)|^{\alpha+2}}\d\tau
\le \alpha M\int_{t_0}^t |y(\tau)|^{\delta} \d\tau \\
&\le \alpha MC_2^\delta \int_{t_0}^t \tau^{-\delta/\alpha}\d\tau
= \frac{\alpha MC_2^\delta}{1-\delta/\alpha}(t^{1-\delta/\alpha}-t_0^{1-\delta/\alpha}).
\end{align*}

Setting $C_3=\alpha^2 MC_2^\delta/(\alpha-\delta)>0$, we obtain
\beq\label{gC}
|g(t)|\le C_3 t^{1-\delta/\alpha}\text{ for all $t\ge t_0$.}
\eeq 

We consider equation \eqref{basicf} as a linear equation of $y$ with time-dependent coefficient $-a|y(t)|^\alpha$ and forcing function $f(t)$. By the variation of constants formula, we solve for $y(t)$ explicitly as
\begin{align*}
y(t)&=e^{-J(t)}\left(y(t_0)+
\int_{t_0}^t e^{J(\tau)}f(\tau)\d\tau\right) \text{ for $t\ge t_0$,}
\end{align*}
where
\beq\label{Iform}
J(t)=a\int_{t_0}^t |y(\tau)|^\alpha\d\tau.
\eeq

Using \eqref{ynorm} in \eqref{Iform}, we rewrite $J(t)$ as 
\begin{align*}
J(t)=\int_{t_0}^t \frac{a}{|y(t_0)|^{-\alpha}+ a\alpha (\tau-t_0)+  g(\tau)}\d\tau=J_1(t)-J_2(t),
\end{align*}
where
\begin{align*}
J_1(t)&=\int_{t_0}^t \frac{a}{|y(t_0)|^{-\alpha}+ a\alpha (\tau-t_0)}\d\tau
\text{ and }
J_2(t)=\int_{t_0}^t  h(\tau) \d\tau,\\
\intertext{ with } 
h(\tau)&=\frac{a g(\tau)}{(|y(t_0)|^{-\alpha}+ a\alpha (\tau-t_0))(|y(t_0)|^{-\alpha}+ a\alpha (\tau-t_0)+ g(\tau))}.
\end{align*}

Clearly,
\beq\label{J1}
J_1(t)=\frac1{\alpha}\ln (1+|y(t_0)|^\alpha a\alpha(t-t_0)).
\eeq

Therefore,
\beq\label{yJform}
y(t)=\frac{e^{J_2(t)}\left( y(t_0) +\int_{t_0}^t e^{J(\tau)}f(\tau)\d\tau\right)}{(1+|y(t_0)|^\alpha a\alpha (t-t_0))^{1/\alpha} }.
\eeq

Consider the integrand $h(\tau)$ of $J_2(t)$. Taking into account the estimate of $|g(\tau)|$ in \eqref{gC}, 
we assert that there is  $t_1\ge t_0$ such that, for $\tau\ge t_1$, 
\begin{align*}
|y(t_0)|^{-\alpha}+ a\alpha (\tau-t_0)&\ge a\alpha \tau/2,\\
\big| |y(t_0)|^{-\alpha}+ a\alpha (\tau-t_0)+g(\tau)\big|
&\ge  a\alpha (\tau-t_0)-C_3\tau^{1-\delta/\alpha}
\ge a\alpha \tau/2.
\end{align*}

Combining these estimates with \eqref{gC} yields, as $\tau\to\infty$,
\beq\label{htau}
|h(\tau)|=\bigo(|g(\tau)|\tau^{-2})=\bigo(\tau^{1-\delta/\alpha}\tau^{-2})=\bigo(\tau^{-1-\delta/\alpha}). 
\eeq

Therefore, 
\beq\label{J2lim} 
\lim_{t\to\infty}J_2(t)=\int_{t_0}^\infty h(\tau)\d\tau=J_*\in\R,
\eeq 
and 
\beqs
J_2(t)=J_*+h_1(t), \text{ where } h_1(t)=-\int_t^\infty h(\tau)\d\tau.
\eeqs

Estimate \eqref{htau} implies 
\beqs
|h_1(t)|=\bigo(t^{-\delta/\alpha}).
\eeqs

Consequently,
\beq\label{eJ2}
e^{J_2(t)}=e^{J_*}e^{h_1(t)}=e^{J_*}+h_2(t),\text{ where } h_2(t)=e^{J_*}(e^{h_1(t))}-1).
\eeq

Since $h_1(t)\to 0$ as $t\to\infty$, we have 
\beq\label{h2}
|h_2(t)|=\bigo(|h_1(t)|)=\bigo(t^{-\delta/\alpha}).
\eeq

Regarding the integral in  formula \eqref{yJform}, we use \eqref{fta}, \eqref{J1} and \eqref{J2lim} to have
\begin{align}
e^{J(t)}|f(t) |
&=(1+|y(t_0)|^\alpha a\alpha( t-t_0))^{1/\alpha} e^{-J_2(t)}|f(t) |\notag \\
&=\bigo(t^{1/\alpha}\cdot 1\cdot t^{-1-1/\alpha-\delta/\alpha})
=\bigo(t^{-1-\delta/\alpha}).\label{eJf}
\end{align}

Therefore,
\beqs
\lim_{t\to\infty}\int_{t_0}^t e^{J(\tau)}f(\tau)\d\tau=\int_{t_0}^\infty e^{J(\tau)}f(\tau)\d\tau=\eta_*\in \R^n,
\eeqs
and
\beq\label{ieJ}
\int_{t_0}^t e^{J(\tau)}f(\tau)\d\tau=\eta_*+\eta(t),\text{ where }\eta(t)=-\int_t^\infty e^{J(\tau)}f(\tau)\d\tau.
\eeq 

It follows \eqref{eJf} that
\beq\label{etes}
|\eta(t)|=\bigo(t^{-\delta/\alpha}).
\eeq

Combining \eqref{yJform}, \eqref{eJ2} and \eqref{ieJ} gives
\beqs
y(t)=\frac{e^{J_*}+h_2(t)}{(1+|y(t_0)|^\alpha a\alpha (t-t_0))^{1/\alpha}}(y(t_0)  +\eta_*+\eta(t))
\text{ for $t\ge t_0$.}
\eeqs

This expression and properties \eqref{h2}, \eqref{etes} imply
\beq \label{yy1}
\left|y(t)-\frac{e^{J_*}(y(t_0)  +\eta_*)}{(1+|y(t_0)|^\alpha a\alpha (t-t_0))^{1/\alpha}} \right|=\bigo(t^{-1/\alpha-\delta/\alpha}).
\eeq

We write
$$\frac{1}{(1+|y(t_0)|^\alpha a\alpha (t-t_0))^{1/\alpha}}=\frac{1}{|y(t_0)|(a\alpha t)^{1/\alpha}}\left(1+\frac{1-|y(t_0)|^\alpha a\alpha t_0}{|y(t_0)|^\alpha a\alpha t}\right)^{-1/\alpha}.$$

We use  the approximation $|(1+x)^{-1/\alpha}-1|=\bigo(|x|)$ as $x\in \R$, $x\to 0$, and then substitute $x=x(t):=\frac{1-|y_0|^\alpha a\alpha t_0}{|y_0|^\alpha a\alpha t}$ when $t$ is large. We obtain
\beqs
\left| \frac{1}{(1+|y(t_0)|^\alpha a\alpha (t-t_0))^{1/\alpha}}
-\frac{1}{|y(t_0)|(a\alpha t)^{1/\alpha}}\right|
=\bigo\left(\frac{|x(t)|}{|y(t_0)|(a\alpha t)^{1/\alpha}}\right)=\bigo(t^{-1/\alpha-1}). 
\eeqs

Combining this fact with \eqref{yy1}, we derive
\beqs
|y(t)-\xi_* t^{-1/\alpha}|=\bigo(t^{-1/\alpha-\delta/\alpha}+t^{-1/\alpha-1})=\bigo(t^{-1/\alpha-\delta/\alpha}),
\eeqs
with
$$\xi_*= \frac{e^{J_*}}{|y(t_0)|(a\alpha)^{1/\alpha}}( y(t_0)+\eta_*)\in\R^n. $$
Therefore, we obtain the desired estimate \eqref{yxi}. 

Because of the lower bound of $|y(t)|$ in \eqref{roughy}, the vector $\xi_*$ in \eqref{yxi} must be nonzeo. 

It remains to prove identity \eqref{axi1}.
 By the triangle inequality and \eqref{yxi}, one has
  \beq\label{yxinorm1}
\big| |y(t)|- |\xi_*| t^{-1/\alpha}\big|=\bigo(t^{-1/\alpha-\varep}).
\eeq
  
From \eqref{ynorm},
\beqs
|y(t)|=\frac1{(a\alpha t)^{1/\alpha}}\cdot \left(1+\frac{|y(t_0)|^{-\alpha}- a\alpha t_0+g(t)}{ a\alpha t}\right)^{-1/\alpha}.
\eeqs
  
Taking into account estimate \eqref{gC} of $|g(t)|$,  we have, as $t\to\infty$,
\beqs
\left| |y(t)|-\frac{t^{-1/\alpha}}{(a\alpha)^{1/\alpha}}\right|
=\bigo\left(\frac1{t^{1/\alpha}}\cdot\frac{\big | |y(t_0)|^{-\alpha}- a\alpha t_0+g(t)\big|}{a\alpha t}\right)
=\bigo\left(\frac1{t^{1/\alpha}}\cdot \frac{|g(t)|}{t}\right),
\eeqs  
which yields
\beq\label{yxinorm2}
\left| |y(t)|-\frac{t^{-1/\alpha}}{(a\alpha)^{1/\alpha}}\right|
=\bigo(t^{-1/\alpha-\delta/\alpha}).
\eeq  

Comparing two asymptotic approximations of the norm $|y(t)|$, as $t\to\infty$, in \eqref{yxinorm1} and \eqref{yxinorm2}, one must have $|\xi_*|=1/(a\alpha)^{1/\alpha}$, which proves \eqref{axi1}.
The proof is complete.
\end{proof}
 
We note from \eqref{yxinorm1} and \eqref{axi1} that the norm $|y(t)|$ can be asymptotically approximated, as $t\to\infty$,  by $|\xi_*|t^{-1/\alpha}$ with  $|\xi_*|=(a\alpha)^{-1/\alpha}$ independent of  the solution $y(t)$. 
This, in fact, agrees with formula \eqref{ynorm}.

\section{Background}\label{aesec}

Consider the ODE system \eqref{yReq} in $\R^n$. We specify the conditions for $A$, $H$ and $G$ in the following.

\begin{assumption}\label{assumpA}
Hereafter,  $A$ is a (real) diagonalizable $n\times n$ matrix with positive eigenvalues.
\end{assumption}

Thanks to Assumption \ref{assumpA}, the spectrum $\sigma(A)$ of matrix $A$ consists of eigenvalues $\Lambda_k$'s, for $1\le k\le n$,  which are positive and increasing in $k$.
Then there exists an invertible $n\times n$ (real) matrix $S$ such that
\beq \label{Adiag}
A=S^{-1} A_0 S,
\text{ where } A_0={\rm diag}[\Lambda_1,\Lambda_2,\ldots,\Lambda_n].
\eeq 

Denote the distinct eigenvalues of $A$ by $\lambda_j$'s which are strictly increasing in $j$, i.e., 
 \beqs
 0<\lambda_1=\Lambda_1<\lambda_2<\ldots<\lambda_{d}=\Lambda_n \text{ for some integer $d\in[1,n]$.}
 \eeqs
 
In the case $A$ is symmetric, the matrix $S$ is orthogonal, i.e., $S^{-1}=S^{\rm T}$, and
\beq\label{xAx}
\Lambda_1|x|^2\le x\cdot Ax\le \Lambda_n |x|^2 \text{ for all $x\in\R^n$.}
\eeq

Regarding the first nonlinearity in \eqref{yReq}, the function $H$ will be in the following class.

\begin{definition}\label{phom} 
Let $X$ and $Y$ be two (real) linear spaces.

A function $F:X\to Y$ is positively homogeneous of degree $\beta>0$ if
\beq\label{Fb}
F(tx)=t^\beta F(x)\text{ for any $x\in X$ and $t>0$.}
\eeq

Define $\mathcal H_\beta(X,Y)$ to be the set of positively homogeneous functions of degree $\beta$ from $X$ to $Y$.

\end{definition}

In Definition \ref{phom}, by taking $x=0$ and $t=2$ in \eqref{Fb}, one has 
$F(0)=0$.
Consequently, \eqref{Fb} holds for all $t\ge 0$.

Clearly, each $\mathcal H_\beta(X,Y)$ is a linear space, and the zero function belongs to $\mathcal H_\beta(X,Y)$ for all $\beta>0$.
If $F:X\to Y$ is a homogeneous polynomial of degree $m\in\N$, then $F\in \mathcal H_m(X,Y)$.
The spaces $\mathcal H_\beta(X,Y)$'s can contain much more complicated functions than polynomials and finite sums of power monomials, see \cite{CaHK1} for many examples.

\begin{assumption}\label{Hf} The function $H$ belongs to $\mathcal H_\alpha(\R^n,\R)$ for some $\alpha>0$,  is continuous on $\R^n$, and $H(x)> 0$ for all $x\in \R^n\setminus \{0\}$.
\end{assumption}

One has from Assumption \ref{Hf} that
\beqs  
0<c_1=\min_{|x|=1} H(x)\le \max_{|x|=1} H(x) =c_2<\infty.
\eeqs 

For $x\in\R^n\setminus\{0\}$, we have $ H(x)=|x|^\alpha H(x/|x|)\in [c_1 |x|^\alpha,c_2 |x|^\alpha]$. Together with the fact $H(0)=0$, this yields
\beq\label{Hyal}
c_1 |x|^\alpha \le H(x)\le c_2 |x|^\alpha \text{ for all }x\in\R^n.
\eeq

 \begin{assumption}\label{assumpF}
We assume the followings.
\begin{enumerate}[label=\tnum]
 \item  The function $x\in \R^n\to H(x)Ax$ is locally Lipschitz.
\item   There is $t_*\ge 0$ such that the function $G(t,x)$ is continuous on $[t_*,\infty)\times \R^n$,  Lipschitz with respect to $x$ on any compact subsets of $[t_*,\infty)\times \R^n$.
\item There exist positive numbers $c_*, r_*, \delta$ such that 
\beq \label{Fcond}
 |G(t,x)| \le c_*|x|^{1+\alpha+\delta} \text{ for all  $t \ge t_*$, and all $x \in \R^n$  with $|x|\le r_*$.}
\eeq  
\end{enumerate}
 \end{assumption}

It follows \eqref{Fcond} that $G(t,0)=0$ for all $t\ge t_*$.
Hence, $y\equiv 0$ on $[t_*,\infty)$ is the trivial solution of \eqref{yReq}.

We first obtain the global existence and uniqueness of solutions of \eqref{yReq} with small initial data, and find their lower and upper bounds for all time.

\begin{theorem}\label{smallthm}
 There is $r_0>0$ such that if $y_0\in \R^n$ satisfies $0<|y_0|<r_0$, then there exists a unique solution $y(t)\in C^1([t_*,\infty),\R^n)$ of \eqref{yReq} on  $(t_*,\infty)$ and $y(t_*)=y_0$.

 Moreover, there are two positive constants  $C_1$ and $C_2$ such that 
 \beq\label{yplus}
 C_1(1+t)^{-1/\alpha}
\le  |y(t)|\le 
 C_2(1+t)^{-1/\alpha} \text{ for all $t\ge t_*$.}
\eeq
 \end{theorem}
\begin{proof}

\textit{Part 1.} Consider $A$ is symmetric first. 
Select $r_0>0$ such that
\beq\label{r0}
2r_0\le r_*\text{ and } c_* (2r_0)^\delta\le  c_1\Lambda_1. 
\eeq

Let $y_0$ be any vector in $\R^n$ with $0<|y_0|<r_0$.
There exists a maximal $T_{\max}\in(t_*,\infty]$ and a unique solution $y(t)\in C^1([t_*,T_{\max}),\R^n)$ of \eqref{yReq} on  $(t_*,T_{\max})$ with $y(t_*)=y_0$ such that
\beq\label{y2r}
|y(t)|< 2r_0\text{ for all $t\in[t_*,T_{\max}).$}
\eeq

We claim that $T_{\max}=\infty$. 
Suppose this is not true, i.e., $T_{\max}<\infty$. Then
\beq\label{lim2r}
\lim_{t\to T_{\max}^-} |y(t)|=2r_0.
\eeq

Taking the dot product of the equation with $y$ gives
\beqs
\frac12 \ddt |y|^2=-H(y)(Ay)\cdot y +G(t,y)\cdot y.
\eeqs 

By \eqref{y2r} and \eqref{r0}, $|y(t)|\le r_*$ on $[t_*,T_{\max})$. Together with \eqref{Fcond}, this implies 
\beq\label{Ginmax}
|G(t,y(t))\cdot y(t)|\le c_*|y(t)|^{2+\alpha+\delta}\text{ for all $t\in[t_*,T_{\max})$.}
\eeq

By \eqref{Hyal}, \eqref{xAx} and \eqref{Ginmax}, one has 
\beqs
\frac12 \ddt |y|^2\le - (c_1 \Lambda_1-c_*|y|^\delta)  |y|^{2+\alpha} \text{ for all $t\in(t_*,T_{\max})$.}
\eeqs

Because of \eqref{y2r} and \eqref{r0}, we have
$$ c_*|y(t)|^\delta \le c_* (2r_0)^\delta \le c_1 \Lambda_1 \text{ for all $t\in [t_*,T_{\max})$.}$$
 
Thus, $\ddt |y(t)|^2\le 0$ on $(t_*,T_{\max})$, which implies 
\beq \label{ydecrease}
|y(t)|\le |y_0|<r_0\text{ for all $t\in[t_*,T_{\max})$.}
\eeq 
This yields a contradiction to \eqref{lim2r}.
Therefore, $T_{\max}=\infty$.

We prove \eqref{yplus} now. 
By the uniqueness/backward uniqueness of the solutions of \eqref{yReq}, one has $y(t)\ne 0$ for all $t\in[t_*,\infty)$.
We calculate, for $t>t_*$,
\beq\label{dyalpha}
\ddt (|y|^{-\alpha})=\alpha \left(\frac{H(y)(Ay)\cdot y}{|y|^{\alpha+2}} - \frac{G(t,y)\cdot y}{|y|^{\alpha+2}}\right).
\eeq

Utilizing \eqref{Hyal}, \eqref{xAx} and \eqref{Ginmax} again, one has 
\beqs
\alpha( c_1 \Lambda_1  - c_* |y|^\delta)\le \ddt  (|y|^{-\alpha})\le  \alpha( c_2 \Lambda_n  +c_*|y|^\delta).
\eeqs

Let $a_1,a_2$ be any positive numbers such that $a_1< c_1\Lambda_1$ and $a_2> c_2\Lambda_n$. 
We select $r_0$ that satisfies the following additional condition
\beq\label{rsmaller}
c_* r_0^\delta\le \min\{  c_1\Lambda_1 -a_1,a_2-c_2\Lambda_n\}. 
\eeq

Together with \eqref{ydecrease}, we obtain 
$\alpha a_1 \le \ddt(|y|^{-\alpha})\le \alpha a_2$ for all $t> t_*$.  
Hence, 
\beq
 \label{yuplow}
\frac{1}{|y_0|^{-\alpha}+ \alpha a_2 (t-t_*)}\le  |y(t)|^\alpha \le \frac{1}{|y_0|^{-\alpha}+ \alpha a_1 (t-t_*)}
 \eeq
for all $t\ge t_*$.
Then the desired estimates in \eqref{yplus} follow.

\textit{Part 2.} Now, consider the general case. Let $A=S^{-1}A_0 S$ as in \eqref{Adiag}.
We make the change of variables $z=Sy$. Note that
\beq\label{Syz}
\|S^{-1}\|_{\rm op}^{-1} |y|\le |z|\le \|S\|_{\rm op}|y|.
\eeq

Then equation \eqref{yReq} with initial condition $y(t_*)=y_0$ is equivalent to
\beq\label{zeq}
z'=-\widetilde H(z)A_0z+\widetilde{G}(t,z), \quad z(t_*)=z_0\eqdef Sy_0,
\eeq
where 
\beq\label{HRz}
\widetilde H(z)=H(S^{-1}z) \text{ and } \widetilde {G}(t,z)=SG(t,S^{-1}z) \text{ for }z\in\R^n.
\eeq

Using the relations in \eqref{Syz}, we can apply the result in Part 1 to equation \eqref{zeq}. If $|z_0|\ne 0$ is sufficiently small, the unique solution $z(t)$ exists for all $t\ge t_*$, and satisfies
\beq\label{zpower}
C_1'(1+t)^{-1/\alpha}
\le  |z(t)|\le 
 C_2'(1+t)^{-1/\alpha}
\text { for all $t\ge t_*$,}
 \eeq
where $C_1',C_2'$ are two positive constants.
Consequently,  when $|y_0|\ne 0$ is sufficiently small, the unique solution $y(t)$ of \eqref{yReq} on $(t_*,\infty)$ with $y(t_*)=y_0$ is $S^{-1}z(t)$, for $t\ge t_*$, and, thanks to \eqref{Syz} and \eqref{zpower}, the estimates in \eqref{yplus} hold true. 
We omit the details.
\end{proof}

As a consequence of \eqref{yplus}, the solution $y(t)$ in Theorem \ref{smallthm} goes to zero as $t\to\infty$. This type of solutions will be the subject of our investigations in the next two sections.

\section{The case of symmetric matrix}\label{symcase}

\textit{Assume, throughout this section, that the matrix $A$ is symmetric.}

Together with Assumption \ref{assumpA}, this implies that the matrix $A$ is positive definite. 
Denote by $A^{1/2}$ the square root matrix of $A$, which is symmetric, positive definite and $(A^{1/2})^2=A$.

For $j=1,2,\ldots,d$, denote by $R_{\lambda_j}$ the orthogonal projection from $\R^n$ to the eigenspace of $A$ associated with the eigenvalue $\lambda_j$.
Then
\beq \label{Rcontract}
|R_{\lambda_j} x|\le |x| \text{ for $j=1,2\ldots,d$ and all $x\in\R^n$.}
\eeq 

Let the function $H$ satisfy Assumption \ref{Hf}.
Assume the following conditions on $G(t,x)$.

\begin{assumption}\label{newG}
The function $G(t,x)$ is continuous on $[t_*,\infty)\times \R^n$ for some $t_*\ge 0$,  and there exist numbers $T_*\ge t_*$ and $c_*, r_*, \delta>0$ such that 
\beq \label{Gcond}
 |G(t,x)| \le c_*|x|^{1+\alpha+\delta} \text{ for all  $t \ge T_*$, and all $x \in \R^n$  with $|x|\le r_*$.}
\eeq  
\end{assumption}

Assume $y(t)$ is a function in $C^1([t_*,\infty),\R^n)$ that solves equation \eqref{yReq} on $(t_*,\infty)$, $y(t)\ne 0$ for all $t\ge t_*$, and
\beq\label{yzerolim}
\lim_{t\to\infty} y(t)= 0.
\eeq

Recall that such a solution $y(t)$ exists under appropriate conditions as shown in Theorem \ref{smallthm}.
Moreover, this theorem essentially gives the following preliminary asymptotic estimates for $|y(t)|$.

\begin{lemma}\label{nzprop}
There exist numbers $\bar T> T_*$ and  $C_1,C_2>0$ such that
\beq\label{ypower}
C_1 t^{-1/\alpha} \le |y(t)|\le C_2 t^{-1/\alpha} \text{ for all $t\ge \bar T$.}
\eeq

Moreover, one has
\beq\label{yinfsup}
 \frac1{(\alpha c_2 \Lambda_n)^{1/\alpha}} \le \liminf_{t\to\infty} t^{1/\alpha}|y(t)|\le\limsup_{t\to\infty} t^{1/\alpha}|y(t)|\le  \frac1{(\alpha c_1 \Lambda_1)^{1/\alpha}}.
\eeq
\end{lemma}
\begin{proof}
Suppose $a_1,a_2>0$ are  two numbers such that $a_1 < c_1\Lambda_1 \le c_2\Lambda_n<a_2$. 

Take $r_0>0$ that satisfies  $r_0\le r_*$ and \eqref{rsmaller}.
Because of the limit \eqref{yzerolim}, there is a number $\bar T> T_*$ such that 
\beqs
|y(t)|<r_0 \text{ for all $t\ge \bar T$.}
\eeqs

Performing the same calculations as in the proof of Theorem \ref{smallthm} from \eqref{dyalpha} to the end of Part 1 with $\bar T$ replacing $t_*$, we obtain, similar to  \eqref{yuplow}, 
\beq\label{newy}
\frac{1}{(|y(\bar T)|^{-\alpha}+ \alpha a_2 (t-\bar T))^{1/\alpha}} 
\le |y(t)|\le 
\frac{1}{(|y(\bar T)|^{-\alpha}+ \alpha a_1 (t-\bar T))^{1/\alpha}}.
 \eeq
  for all $t\ge \bar T$. Because $\bar T>0$ and $|y(\bar T)|>0$, we derive from \eqref{newy} the estimates in \eqref{ypower} for some positive constants $C_1,C_2$ depending on $a_1,a_2$.
By the particular choice $a_1=c_1\Lambda_1/2$ and $a_2=c_2\Lambda_n+1$, the numbers $\bar T$ and $C_1,C_2$ are now fixed in \eqref{ypower}.

It follows \eqref{newy} that
\beqs
 \frac1{(\alpha a_2)^{1/\alpha}} \le \liminf_{t\to\infty} t^{1/\alpha}|y(t)|\le\limsup_{t\to\infty} t^{1/\alpha}|y(t)|\le  \frac1{(\alpha a_1)^{1/\alpha}},
\eeqs
for any numbers $a_1\in(0,c_1 \Lambda_1)$ and $a_2\in(c_2 \Lambda_n,\infty)$.
Thus, we obtain \eqref{yinfsup}. 
\end{proof}

As a consequence of \eqref{ypower} and \eqref{Hyal}, one has, for all $t\ge \bar T$,
\beq\label{Hyt}
C_3 t^{-1}\le H(y(t)) \le C_4 t^{-1}, \text{ where $C_3=c_1C_1^\alpha$ and $C_4=c_2C_2^\alpha$.} 
\eeq

For $t\ge t_*$, define
\beq\label{quot}
\lambda(t)=\frac{|A^{1/2}y(t)|^2}{|y(t)|^2}=\frac{y(t)\cdot Ay(t)}{|y(t)|^2} \text{ and }   
 v(t)=\frac{y(t)}{|y(t)|} .
\eeq

Then $\lambda\in C^1([t_*,\infty))$ and, thanks to \eqref{xAx},  
 \beq\label{lest}
 \Lambda_1 \le  \lambda(t)\le \Lambda_n\le \|A\| \text{ for all $t\ge t_*$.}
 \eeq 
  Hence,
 \beq\label{llimbound}
\Lambda_1 \le \liminf_{t\to\infty} \lambda(t)\le \limsup_{t\to\infty}\lambda(t)\le \Lambda_n.
 \eeq
 
 Also, $v\in C^1([t_*,\infty),\R^n)$ and 
 \beq\label{vone}
 |v(t)|=1\text{ for all $t\ge t_*$.}
 \eeq
 
\begin{lemma}\label{lem1}
The quotient $\lambda(t)$ converges, as $t\to\infty$, to an eigenvalue $\Lambda$ of $A$.
\end{lemma}
\begin{proof}
We find a differential equation for $\lambda(t)$. Using the fact that $A$ is symmetric, we have, for $t>t_*$,
\begin{align*}
\lambda'(t)
&=\frac1{|y|^2} \ddt |A^{1/2}y|^2-\frac{|A^{1/2}y|^2}{|y|^4}\ddt |y|^2
=\frac2{|y|^2} y' \cdot Ay-\frac{2|A^{1/2}y|^2}{|y|^4} y' \cdot y.
\end{align*}

Hence,
\beq\label{lameq}
\lambda'(t)=\frac2{|y|^2} y' \cdot (Ay-\lambda y).
\eeq

We rewrite equation \eqref{yReq} as
\beqs
y' =-H(y)(Ay-\lambda y) - \lambda H(y)y+G(t,y).
\eeqs

Using this expression of $y'$ in \eqref{lameq} yields
\beq\label{prelh}
\lambda'(t)
=-\frac{2H(y)}{|y|^2} |Ay-\lambda y|^2 - \frac{2\lambda H(y)}{|y|^2} y\cdot (Ay-\lambda y)+h(t),
\eeq
where
\beqs
h(t)=\frac2{|y(t)|^2}G(t,y(t))\cdot(Ay(t)-\lambda(t) y(t)).
\eeqs

Note that the second term on the right-hand side of \eqref{prelh} vanishes thanks to the fact
$ y\cdot(Ay-\lambda y)=0$.
Then equation \eqref{prelh}  reduces to
\beq\label{lh}
\lambda'(t)
=-2H(y)|Av-\lambda v|^2 +h(t).
\eeq

Using \eqref{Gcond}, \eqref{lest}, \eqref{vone} and, then,  \eqref{ypower}, we estimate  
\beq\label{oh}
|h(t)|\le 2c_*|y(t)|^{\alpha+\delta} (2 \|A\|)\le C_5 t^{-1-\delta/\alpha},
\eeq
for all $t\ge \bar T$, where $C_5=4c_*\|A\| C_2^{\alpha+\delta}$. 

For $t'>t\ge \bar T$, integrating equation \eqref{lh} from $t$ to $t'$ gives
\beq\label{llH}
\lambda(t')-\lambda(t)+2\int_t^{t'}H(y(\tau))|Av(\tau)-\lambda(\tau) v(\tau)|^2\d\tau =\int_t^{t'}h(\tau)\d\tau.
\eeq

Note from \eqref{oh} that 
\beqs
\left|\int_t^{t'}h(\tau)\d\tau\right|\le \frac{\alpha C_5}{\delta}t^{-\delta/\alpha}.
\eeqs

Thus, taking the limit superior of \eqref{llH}, as $t'\to\infty$, yields
\beq\label{llH2}
\limsup_{t'\to\infty}\lambda(t')\le \lambda(t)+\frac{\alpha C_5}{\delta}t^{-\delta/\alpha}<\infty.
\eeq

Then taking the limit inferior of \eqref{llH2}, as $t\to\infty$, gives
\beqs
\limsup_{t'\to\infty}\lambda(t')\le \liminf_{t\to\infty}  \lambda(t).
\eeqs

This and \eqref{llimbound} imply $\lambda(t)$ converges as $t\to\infty$, and 
\beq\label{limlam} 
\lim_{t\to\infty}\lambda(t)=\Lambda\in [\Lambda_1,\Lambda_n].
\eeq

Next, we prove that  $\Lambda$ is an eigenvalue of $A$.
Using property \eqref{limlam} in \eqref{llH} and the Cauchy criterion, as $t,t'\to\infty$, we obtain
\beq\label{Hfin}
\int_{t_*}^\infty H(y(\tau))|Av(\tau)-\lambda(\tau) v(\tau)|^2\d\tau<\infty.
\eeq

\textbf{Claim.} $\forall\varep>0,\exists t\ge\max\{\bar T,1/\varep\}:|Av(t)-\lambda(t)v(t)|<\varep$.

We accept this claim momentarily. Then there exists a sequence $t_j\to\infty$, as $j\to\infty$, such that 
\beq \label{Avj}
\lim_{j\to\infty}|Av(t_j)-\lambda(t_j)v(t_j)|= 0.
\eeq 

Because of \eqref{vone} and by taking a subsequence of $(t_j)_{j=1}^\infty$, we can assume $v(t_j)\to \bar v\in\R^n$, as $j\to\infty$,  with $|\bar v|=1$. We already have from \eqref{limlam} that $\lambda(t_j)\to \Lambda$. 
Then the limit  \eqref{Avj} gives
$A\bar v=\Lambda\bar v$.
Thus, $\Lambda$ is an eigenvalue of $A$.

Finally, we prove the Claim. Suppose the Claim is not true. Then 
\beq\label{anticlaim} 
\exists\varep_0>0,\forall t\ge \max\{\bar T,1/\varep_0\}: |Av(t)-\lambda(t)v(t)|\ge \varep_0.
\eeq

Let $T=\max\{\bar T,1/\varep_0\}$. Combining \eqref{anticlaim} with property \eqref{Hyt}, we have
\beqs
\int_{T}^\infty H(y(\tau))|Av(\tau)-\lambda(\tau) v(\tau)|^2\d\tau
\ge \int_{T}^\infty C_3\tau^{-1} \varep_0^2 \d\tau=\infty,
\eeqs
which contradicts \eqref{Hfin}. Hence, the  Claim is true and the proof of Lemma \ref{lem1} is complete.
\end{proof}

\emph{For the remainder of this section, $\Lambda$ is the eigenvalue in Lemma \ref{lem1}.}

\begin{lemma}\label{lem2}
There is $\varep>0$ such that 
\beq\label{remv}
| ({\rm Id}-R_{\Lambda})v(t)|=\bigo(t^{-\varep}) \text{ as $t\to\infty$.}
 \eeq
\end{lemma}
\begin{proof}
We find a differential  equation for $v(t)$. We compute
\begin{align*}
 v' &=\frac1{|y|}y' -\frac1{|y|^3}(y' \cdot y)y\\
 &=\frac1{|y|}(-H(y)Ay+G(t,y))-\frac1{|y|^3}((-H(y)Ay+G(t,y))\cdot y)y\\
 &=-H(y)Av+H(y)|A^{1/2}v|^2v +g(t),
\end{align*}
where $g:[t_*,\infty)\to\R^n$ is defined by
\beqs
g(t)=\frac{1}{|y(t)|}G(t,y(t)) -\frac{G(t,y(t))\cdot y(t)}{|y(t)|^3}y(t).
\eeqs

Thanks to the fact $|A^{1/2}v|^2=\lambda(t)$, we have
\beq \label{dv}
 v'=-H(y)(Av-\lambda v)+g(t)\text{ for all $t>t_*$.}
\eeq 

Using property \eqref{Gcond} of $G(t,y)$, one can estimate
\beq\label{gMy}
|g(t)|\le 2 c_* |y(t)|^{\alpha+\delta}\text{ for all $t\ge T_*$.}
\eeq

We write 
\beq\label{InR} 
({\rm Id}-R_{\Lambda})v=\sum_{1\le j\le d, \lambda_j\ne  \Lambda } R_{\lambda_j}v.
\eeq
We will estimate each $|R_{\lambda_j}v(t)|$ on the right-hand side of \eqref{InR}. 

Define $\mu=\min \{ |\lambda_j-\Lambda| : 1\le j\le d, \lambda_j\ne  \Lambda\}>0$.

Let $\lambda_j\in\sigma(A)\setminus\{\Lambda\}$.
Applying $R_{\lambda_j}$ to equation \eqref{dv} and taking the dot product with $R_{\lambda_j}v$ yield
\beq\label{Rnorm}
\frac12 \ddt |R_{\lambda_j}v|^2
 =-H(y)(\lambda_j-\lambda)| R_{\lambda_j}v|^2 +R_{\lambda_j} g(t)\cdot R_{\lambda_j}v.
\eeq

By  Cauchy--Schwarz's inequality, then Cauchy's inequality, and property \eqref{Hyal}, we have
\begin{align*}
|R_{\lambda_j}g(t)\cdot R_{\lambda_j}v|
&\le 2c_*  |y|^{\alpha+\delta}|R_{\lambda_j}v|
\le \frac\mu4 H(y)| R_{\lambda_j}v|^2 + \frac{4c_*^2|y|^{2\alpha+2\delta}}{\mu H(y)}\\
&\le \frac\mu4 H(y)| R_{\lambda_j}v|^2 + \frac{4c_*^2}{\mu c_1}|y|^{\alpha+2\delta}.
\end{align*}

Together with the use of \eqref{ypower} to estimate the last norm $|y(t)|$, this implies, for $t\ge \bar T$,
\beq\label{abso}
|R_{\lambda_j}g(t)\cdot R_{\lambda_j}v|
\le \frac\mu4 H(y)| R_{\lambda_j}v|^2 +\frac{C_6}2 t^{-1-2\delta/\alpha},
\text{ with } C_6= \frac{8c_*^2 C_2^{\alpha+2\delta}}{\mu c_1}.
\eeq

\medskip
\noindent\underline{\textit{Case $\lambda_j>\Lambda$.}} In this case, combining \eqref{Rnorm} and \eqref{abso} yields, for $t\ge \bar T$,
\begin{align*}
\frac12 \ddt |R_{\lambda_j}v|^2
 &\le -(\lambda_j-\lambda-\frac\mu4) H(y)| R_{\lambda_j}v|^2 +\frac{C_6}2 t^{-1-2\delta/\alpha} .
\end{align*}

Note, as $t\to\infty$, that  $$\lambda_j-\lambda(t)-\frac\mu4\to \lambda_j-\Lambda-\frac\mu4\ge \mu-\frac\mu4=\frac{3\mu}4.$$ 

One has  $\lambda_j-\lambda(t)-\frac\mu4 \ge \frac\mu2$  for $t$ sufficiently large.
Thus, for sufficiently large $t$,
\beq\label{dRlarge}
 \ddt |R_{\lambda_j}v|^2
 \le - \mu H(y)| R_{\lambda_j}v|^2 +C_6t^{-1-2\delta/\alpha} .
\eeq

Below, $T\in[\bar T,\infty)$ is fixed and can be taken sufficiently large. 

Let $t$ and $t_0$ be any numbers in $[T,\infty)$ with  $t>t_0$.
 It follows \eqref{dRlarge} that 
\beq\label{RHint}
  |R_{\lambda_j}v(t)|^2
 \le e^{-\mu\int_{t_0}^t H(y(\tau))\d\tau} | R_{\lambda_j}v(t_0)|^2 + C_6\int_{t_0}^t e^{-\mu\int_{\tau}^t H(y(s))\d s}\tau^{-1-2\delta/\alpha}\d\tau .
\eeq

Regarding the first inequality in \eqref{Hyt}, we fix  a number $\theta>0$ such that 
$$\theta\le C_3\text{ and } \theta\mu<2\delta/\alpha.$$ 
Then 
\beq \label{Hthe}
H(y(t))\ge \theta t^{-1}\text{ for all $t\ge T$.}
\eeq 

Utilizing this estimate in \eqref{RHint} gives, for $t\ge T$,
\begin{align*}
  |R_{\lambda_j}v(t)|^2
 &\le e^{-\theta\mu\int_{t_0}^t \tau^{-1}\d\tau} | R_{\lambda_j}v(t_0)|^2 
 + C_6\int_{t_0}^t e^{-\theta\mu\int_{\tau}^t s^{-1}\d s}\tau^{-1-2\delta/\alpha}\d\tau \\
  &= \frac{t_0^{\theta\mu}}{t^{\theta\mu}} | R_{\lambda_j}v(t_0)|^2 
  + C_6 \int_{t_0}^t \frac{\tau^{\theta\mu}}{t^{\theta\mu}}\tau^{-1-2\delta/\alpha}\d\tau \\
    &= \frac{t_0^{\theta\mu}}{t^{\theta\mu}} | R_{\lambda_j}v(t_0)|^2 
    +\frac{C_6}{t^{\theta\mu}(2\delta/\alpha-\theta\mu)}\left(  t_0^{\theta\mu-2\delta/\alpha}-t^{\theta\mu-2\delta/\alpha}\right).
\end{align*}

Thus,
\beq\label{Rlamv1}
  |R_{\lambda_j}v(t)|^2
    \le \frac{t_0^{\theta\mu}}{t^{\theta\mu}} | R_{\lambda_j}v(t_0)|^2 +\frac{C_6t_0^{\theta\mu-2\delta/\alpha}}{t^{\theta\mu}(2\delta/\alpha-\theta\mu)}.
\eeq

With $t_0$ fixed in \eqref{Rlamv1}, we obtain
\beq \label{Rv1}
|R_{\lambda_j}v(t)|=\bigo(t^{-\theta\mu/2}) \text{ as $t\to\infty$.}
\eeq 

\medskip
\noindent\underline{\textit{Case $\lambda_j<\Lambda$.}} Using \eqref{abso} to have a lower bound for  the last term in \eqref{Rnorm}, we have
\begin{align*}
\frac12 \ddt |R_{\lambda_j}v|^2
 &\ge (\lambda-\lambda_j-\frac\mu4) H(y)| R_{\lambda_j}v|^2 -\frac{C_6}2 t^{-1-2\delta/\alpha} .
\end{align*}

As $t\to\infty$,
$$\lambda(t)- \lambda_j -\frac\mu4\to \Lambda - \lambda_j -\frac\mu4\ge \mu-\frac\mu4=\frac{3\mu}4.$$ 

Thus, for  sufficiently large $t$, one has $\lambda(t)-\lambda_j-\frac\mu4 \ge \frac\mu2$, and hence,
\begin{align*}
\ddt |R_{\lambda_j}v|^2
 &\ge \mu H(y)| R_{\lambda_j}v|^2 -C_6t^{-1-2\delta/\alpha} .
\end{align*}

Again, we can take a number $T\in[\bar T,\infty)$ sufficiently large in the calculations below. 
Then, for any $t,t_0$ such that $t>t_0\ge T$, one has
\beq\label{eHR}
e^{-\mu\int_{t_0}^t H(y(\tau))\d\tau}  |R_{\lambda_j}v(t)|^2-  |R_{\lambda_j}v(t_0)|^2
 \ge -C_6 \int_{t_0}^t e^{-\mu\int_{t_0}^\tau H(y(s))\d s}\tau^{-1-2\delta/\alpha} \d\tau.
\eeq

Note that $\int_{t_0}^\infty H(y(\tau))\d\tau=\infty$ and $ |R_{\lambda_j}v(t)|\le |v(t)|=1$. Then
\beqs
\lim_{t\to\infty} e^{-\mu\int_{t_0}^t H(y(\tau))\d\tau}  |R_{\lambda_j}v(t)|^2=0.
\eeqs

Letting $t\to\infty$ in \eqref{eHR} and using \eqref{Hthe}  yield
\begin{align*}
|R_{\lambda_j}v(t_0)|^2 
 &\le C_6 \int_{t_0}^\infty e^{-\mu\int_{t_0}^\tau H(y(s))\d s}\tau^{-1-2\delta/\alpha} \d\tau\le C_6 \int_{t_0}^\infty \frac{t_0^{\theta\mu}}{\tau^{\theta\mu}}\tau^{-1-2\delta/\alpha} \d\tau
 = \frac{C_6}{\theta\mu+2\delta/\alpha} t_0^{-2\delta/\alpha}.
\end{align*}

Therefore, we obtain
\beq\label{Rv2} 
|R_{\lambda_j}v(t_0)| =\bigo(t_0^{-\delta/\alpha})\text{ as $t_0\to\infty$.}
\eeq 

\medskip
By the expression \eqref{InR} of $({\rm Id}-R_\Lambda)v(t)$, estimate \eqref{Rv1} of $|R_{\lambda_j}v(t)|$ for all  $\lambda_j>\Lambda$, and  estimate \eqref{Rv2} of $|R_{\lambda_j}v(t)|$  for all   $\lambda_j<\Lambda$, we obtain  \eqref{remv} with $\varep=\min\{\theta\mu/2,\delta/\alpha\}=\theta\mu/2$.
\end{proof}

Lemma \ref{lem2} results in the following estimates for $y(t)$ which refine \eqref{ypower}.

\begin{corollary}\label{cor1}
Let $\varep>0$ be as in Lemma \ref{lem2}. Then 
\beq\label{remy}
|({\rm Id}-R_\Lambda)y(t)|=\bigo(t^{-1/\alpha-\varep})\text{ as $t\to\infty$,}
\eeq
and there exist numbers $T_0\ge \bar T$ and $C_7,C_8>0$ such that
\beq\label{RLy}
C_7 t^{-1/\alpha}\le |R_\Lambda y(t)|\le C_8 t^{-1/\alpha} \text{ for all $t\ge T_0$.}
\eeq
\end{corollary}
\begin{proof}
On the one hand, we have from \eqref{ypower} and \eqref{remv} that
\beqs
|({\rm Id}-R_\Lambda)y(t)|=|y(t)|\cdot  |({\rm Id}-R_\Lambda)v(t)|=\bigo(t^{-1/\alpha}\cdot t^{-\varep}),
\eeqs
which proves \eqref{remy}. On the other hand, by the triangle inequality and \eqref{ypower}, one has
\begin{align*}
|R_\Lambda y(t)|&\le |y(t)|+|({\rm Id}-R_\Lambda)y(t)|\le C_2 t^{-1/\alpha}+|({\rm Id}-R_\Lambda)y(t)|,\\
|R_\Lambda y(t)|&\ge |y(t)|-|({\rm Id}-R_\Lambda)y(t)|\ge  C_1 t^{-1/\alpha}-|({\rm Id}-R_\Lambda)y(t)|.
\end{align*}

These inequalities and estimate \eqref{remy} for $|({\rm Id}-R_\Lambda)y(t)|$ imply the desired lower and upper bounds for $|R_\Lambda y(t)|$ in \eqref{RLy} when $t$ is sufficiently large.
\end{proof}

\begin{lemma}\label{lem4}
There exists a unit vector $v_*\in\R^n$ such that
\beq\label{RLv}
|R_{\Lambda}v(t)-v_*|=\bigo(t^{-\varep}),\text{ as $t\to\infty$, for some $\varep>0$,}
\eeq
\beq\label{vlim}
|v(t)-v_*|=\bigo(t^{-\varep}), \text{ as $t\to\infty$, for some $\varep>0$.}
\eeq

Consequently, one has
\beq\label{Rvlim}
\lim_{t\to\infty} R_{\Lambda}v(t)=\lim_{t\to\infty} v(t)=v_*.
\eeq 
\end{lemma}
\begin{proof}
We prove \eqref{RLv} first. 
Let $\varep_0>0$ be such that \eqref{remv} holds for $\varep=\varep_0$.
Then one has
\beq\label{oneR}
\big| 1-|R_{\Lambda}v(t)| \big|
=\big| |v(t)|-|R_{\Lambda}v(t)|\big|\le |v(t)-R_{\Lambda}v(t)|=\bigo(t^{-\varep_0}).
\eeq

As a consequence of \eqref{oneR},  $ |R_{\Lambda}v(t)|\to 1$ as $t\to\infty$. 
Hence, $R_{\Lambda}v(t)\ne 0$ for large $t$.

Applying $R_\Lambda$ to equation \eqref{dv} yields, for  $t>t_*$,
\beq \label{dRv}
\ddt R_{\Lambda}v
 =-H(y)(\Lambda-\lambda) R_{\Lambda}v +R_\Lambda g(t).
\eeq

Then,  for large $t$,
\beq\label{RLveq}
\ddt |R_{\Lambda}v|
=\frac{1}{|R_{\Lambda}v|} (\ddt R_{\Lambda}v)\cdot R_{\Lambda}v
 =-H(y)(\Lambda-\lambda)| R_{\Lambda}v| +g_1(t),
\eeq
where
\beqs
g_1(t)=\frac1{|R_{\Lambda}v(t)|}R_\Lambda  g(t)\cdot R_{\Lambda}v(t).
\eeqs

Consider $T\in[\bar T,\infty)$  sufficiently large.
Solving for solution $ |R_{\Lambda}v(t)|$ by the variation of constants formula from the  differential equation \eqref{RLveq} gives, for $t>t_0\ge T$, 
\beqs
 |R_{\Lambda}v(t)|=e^{-\int_{t_0}^t H(y(\tau))(\Lambda-\lambda(\tau))\d\tau } \left(| R_{\Lambda}v(t_0) |
 + \int_{t_0}^t e^{\int_{t_0}^\tau H(y(s))(\Lambda-\lambda(s))\d s}  g_1(\tau) \d\tau \right).
\eeqs

It yields
\beq \label{hl0}
\begin{aligned}
&\int_{t_0}^t H(y(\tau))(\Lambda-\lambda(\tau))\d\tau\\
&=\ln \left(| R_{\Lambda}v(t_0) |
 + \int_{t_0}^t e^{\int_{t_0}^\tau H(y(s))(\Lambda-\lambda(s))\d s}  g_1(\tau) \d\tau \right)
   - \ln |R_{\Lambda}v(t)|.
\end{aligned}
\eeq

We have from \eqref{Rcontract}, \eqref{gMy} and \eqref{ypower} that, for $t\ge T$,
\beq\label{gg}
|g_1(t)|\le|R_\Lambda  g(t)| \le  |g(t)|\le C_9 t^{-(1+\delta/\alpha)}, \text{ where $C_9=2c_*C_2^{\alpha+\delta}$.}
\eeq

Recall that $C_4$ is the positive constant in \eqref{Hyt}.
We  take $\varep_1>0$ small such that
\beq\label{cep1}
C_4\varep_1<\delta/\alpha.
\eeq

Thanks to Lemma \ref{lem1}, we can assume $T$ is sufficiently large so that 
\beqs
|\Lambda-\lambda(s)|<\varep_1 \text{ for all $s\ge T$.}
\eeqs

Together with \eqref{gg}, we have, for $\tau\in[t_0,t]$,
\begin{align*}
e^{\int_{t_0}^\tau H(y(s))(\Lambda-\lambda(s))\d s} | g_1(\tau) |
&\le e^{\int_{t_0}^\tau C_4\varep_1 s^{-1}\d s} C_9\tau^{-(1+\delta/\alpha)}
=\frac{ \tau^{C_4\varep_1}}{ t_0^{C_4\varep_1}} C_9\tau^{-(1+\delta/\alpha)}\\
&=\frac{C_9}{ t_0^{C_4\varep_1}} \tau^{-1-\delta/\alpha+C_4\varep_1}.
\end{align*}

Thanks to this and \eqref{cep1}, 
\beq\label{eta0}
\lim_{t\to\infty}\int_{t_0}^t e^{\int_{t_0}^\tau H(y(s))(\Lambda-\lambda(s))\d s}  g_1(\tau) \d\tau
=\int_{t_0}^\infty e^{\int_{t_0}^\tau H(y(s))(\Lambda-\lambda(s))\d s}  g_1(\tau) \d\tau
=\eta(t_0)\in \R.
\eeq

Note that 
\beq\label{eta}
|\eta(t_0)|\le  \frac{C_9}{ t_0^{C_4\varep_1}} \int_{t_0}^\infty\tau^{-1-\delta/\alpha+C_4\varep_1}\d\tau
=\frac{C_9}{\delta/\alpha-C_4\varep_1} t_0^{-\delta/\alpha}.
\eeq

Passing to the limit as $t\to\infty$ in \eqref{hl0}, we have
\beq\label{hl1}
\int_{t_0}^\infty H(y(\tau))(\Lambda-\lambda(\tau))\d\tau= \ln (| R_{\Lambda}v(t_0) |
 + \eta(t_0))-\ln 1\in\R.
\eeq

By \eqref{hl1}, we can define, for $t_0\ge T$,
$$h(t_0)=\int_{t_0}^\infty H(y(\tau))(\Lambda-\lambda(\tau))\d\tau\in\R.$$

We rewrite \eqref{hl1} as
\beqs
h(t_0)=\ln (| R_{\Lambda}v(t_0) | + \eta(t_0))=\ln (1+(| R_{\Lambda}v(t_0) | -1)+ \eta(t_0)).
\eeqs

With this expression and properties \eqref{oneR} and \eqref{eta}, we have, as $t_0\to\infty$, 
\beq\label{hl2}
| h(t_0) |
=\bigo(\big|| R_{\Lambda}v(t_0) | -1\big|+ |\eta(t_0)|)
=\bigo(t_0^{-\varep_0}+t_0^{-\delta/\alpha})
=\bigo(t_0^{-\varep_2}),
\eeq
where $\varep_2=\min\{\varep_0,\delta/\alpha\}$.

Solving for $R_{\Lambda}v(t)$ from \eqref{dRv} by the variation of constants formula, one has
\beq\label{RvH}
R_{\Lambda}v(t)
 =e^{-\int_{t_0}^t H(y(\tau))(\Lambda-\lambda(\tau))\d\tau } \left( R_{\Lambda}v(t_0) 
 + \int_{t_0}^t e^{\int_{t_0}^\tau H(y(s))(\Lambda-\lambda(s))\d s}  R_{\Lambda}g(\tau)\d\tau \right) .
\eeq

Recall $|R_{\Lambda}g(\tau)|\le |g(\tau)|$.
Same arguments as those from \eqref{gg} to \eqref{eta} with $R_\Lambda g(\tau)$ replacing $g_1(\tau)$, we obtain, similar to \eqref{eta0} and \eqref{eta}, that
\beqs
\lim_{t\to\infty} \int_{t_0}^t e^{\int_{t_0}^\tau H(y(s))(\Lambda-\lambda(s))\d s}  R_{\Lambda}g(\tau) \d\tau
= \int_{t_0}^\infty e^{\int_{t_0}^\tau H(y(s))(\Lambda-\lambda(s))\d s}  R_{\Lambda}g(\tau) \d\tau=X(t_0)\in\R^n
\eeqs
for $t_0\ge T$, and  
 \beq \label{oX}
 |X(t_0)|=\bigo(t_0^{-\varep_2})\text{ as $t_0\to\infty$.}
\eeq

Taking $t\to\infty$ in \eqref{RvH} gives
\beqs
\lim_{t\to\infty}R_{\Lambda}v(t)=v_*\eqdef e^{-h(t_0)}( R_{\Lambda}v(t_0) +X(t_0))\in\R^n.
\eeqs

We rewrite \eqref{RvH} as
\beqs
R_{\Lambda}v(t)
=e^{h(t)-h(t_0)}\left (R_{\Lambda}v(t_0)+X(t_0)-\int_t^\infty e^{h(t_0)-h(\tau)} R_\Lambda g(\tau)\d\tau\right)=e^{h(t)}v_*-X(t).
\eeqs

Thus,
\beqs
|R_{\Lambda}v(t)-v_*|
\le |e^{h(t)}-1|\cdot |v_*| + |X(t)|.
\eeqs

Having $t_0$ fixed in the above formula and using \eqref{hl2} and \eqref{oX},  we deduce, as $t\to\infty$,
\beqs
|R_{\Lambda}v(t)-v_*|
=\bigo(|e^{h(t)}-1|+|X(t)|)
=\bigo(|h(t))|+|X(t)|)=\bigo(t^{-\varep_2}).
\eeqs
Therefore, we obtain the desired estimate \eqref{RLv}.

Next, by the triangle inequality,
\beqs
|v(t)-v_*|\le |v(t)-R_\Lambda v(t)|+|R_\Lambda v(t)-v_*|.
\eeqs
This inequality and estimates \eqref{remv}, \eqref{RLv} imply \eqref{vlim}.

Finally, \eqref{Rvlim} is a direct consequence of \eqref{RLv} and \eqref{vlim}.
The proof is complete.
\end{proof}

To derive further properties of $y(t)$, we require more conditions on the function $H$.

\begin{definition}\label{HCdef}
Let $E$ be a nonempty subset of $\R^n$ and $F$ be a function from $E$ to $\R$. We say $F$ has property (HC) on $E$ if, 
for any  $x_0\in E$,  there exist  numbers  $r,C,\gamma>0$ such that
\beq\label{FHder}
|F(x)-F(x_0)|\le C|x-x_0|^\gamma
\eeq
for any $x\in E$ with $|x-x_0|<r$.
\end{definition}

Because the power $\gamma$ is allowed to depend on each $x_0$, the property (HC) is weaker than a H\"older continuity requirement such as $F\in C^{0,\gamma}(\bar E)$ for some $\gamma\in(0,1)$.

The next condition imposed on $H$ is the following.

\begin{assumption}\label{HHsphere}
The function $H$ has property (HC) on the unit sphere $\{x\in\R^n:|x|=1\}$.
\end{assumption}

The asymptotic behavior of $y(t)$ can be described in the theorem below.

\begin{theorem}\label{thmsym}
There exists a nonzero vector $\xi_*\in\R^n$ such that, as $t\to\infty$,
\beq\label{ysim}
|y(t)-\xi_* t^{-1/\alpha}|=\bigo(t^{-1/\alpha-\varep})\text{ for some $\varep>0$.}
\eeq

In fact, $\xi_*$ is an eigenvector of $A$ associated with the eigenvalue $\Lambda$, and one has
\beq\label{Rxi}
|R_\Lambda y(t)-\xi_* t^{-1/\alpha}|=\bigo(t^{-1/\alpha-\varep})\text{ for some $\varep>0$.}
\eeq

Moreover,
\beq\label{xiHA}
\alpha \Lambda H(\xi_*)=1.
\eeq
\end{theorem}
\begin{proof}
We prove  \eqref{Rxi} first. We look for a differential equation for $R_\Lambda y(t)$.

Applying $R_\Lambda$ to equation \eqref{yReq} and rewriting $H(y)=|y|^\alpha H(v)$, we have
\beq\label{Ry1}
(R_\Lambda y)' =-\Lambda  |y|^{\alpha}H(v)R_\Lambda y+R_\Lambda G(t,y).
\eeq

Let $v_*$ be the unit vector in Lemma \ref{lem4}, and $\varep_0>0$ be such that \eqref{remv}, \eqref{remy}, \eqref{RLv}, and \eqref{vlim} hold for $\varep=\varep_0$.

We approximate $|y|^\alpha H(v)$ on the right-hand side of \eqref{Ry1}  in the following way
\beqs
|y|^\alpha H(v)=|R_\Lambda y|^\alpha H(v_*)+g_1(t),
\eeqs
where
\begin{align*}
g_1(t)&=|y(t)|^\alpha (H(v(t))-H(v_*))+(|y(t)|^\alpha -|R_\Lambda y(t)|^\alpha)H(v_*)\\
&=|y(t)|^\alpha \big\{ H(v(t))-H(v_*)+(1 -|R_\Lambda v(t)|^\alpha)H(v_*)\big\}.
\end{align*}

Then 
\beq\label{Rygood}
(R_\Lambda y)' 
= -\Lambda H(v_*) |R_\Lambda y|^{\alpha}R_\Lambda y+g(t),
\eeq
where
$g(t)=-\Lambda g_1(t)R_\Lambda y(t) +R_\Lambda G(t,y(t))$.

We estimate $|g_1(t)|$ first and then $|g(t)|$. 
Take $F=H$, $E$=the unit sphere, the unit vector $x_0=v_*$  in Definition \ref{HCdef}.
Then there exists a number $\gamma>0$ depending on $v_*$  such that, according to \eqref{FHder} with the unit vector $x=v(t)$ for sufficiently large $t$,
\beq\label{Hvv}
|H(v(t))-H(v_*)|=\bigo(|v(t)-v_*|^\gamma) =\bigo(t^{-\gamma\varep_0}).
\eeq

Recall that $\lim_{t\to\infty} |R_\Lambda v(t)|=1$. 
We use the approximation $|s^\alpha-1|=\bigo(|s-1|)$ when the real number $s\to 1$.
As $t\to\infty$, by taking $s=|R_\Lambda v(t)|$ and using estimate  \eqref{oneR}, we derive
\beq\label{1Ral}
\big|1 -|R_\Lambda v(t)|^\alpha\big|
=\bigo\left(\big |1-|R_\Lambda v(t)|\big| \right)=\bigo(t^{-\varep_0}).
\eeq

Combining \eqref{Hvv}, \eqref{1Ral} with \eqref{ypower}, we obtain
\beq\label{g1}
|g_1(t)|=\bigo(|y(t)|^\alpha (t^{-\gamma\varep_0}+t^{-\varep_0}))=\bigo(t^{-1-\varep_1}),
\eeq
where $\varep_1=\varep_0\min\{1,\gamma\}$.

For the last term in $g(t)$, we have from \eqref{Gcond} and \eqref{ypower} that
\beq\label{Ry5}
|R_\Lambda G(t,y(t))|=\bigo( |y(t)|^{\alpha+1+\delta})=\bigo(t^{-1-1/\alpha-\delta/\alpha}).
\eeq

Combining \eqref{g1} and \eqref{Ry5} gives
\beqs
|g(t)|=\bigo(t^{-1-\varep_1}|R_\Lambda y(t)|+t^{-1-1/\alpha-\delta/\alpha})
=\bigo(t^{-1-1/\alpha-\varep_2/\alpha}),
\eeqs
where $\varep_2=\min\{\varep_1\alpha ,\delta\}$.

By the lower bound of $|R_\Lambda y(t)|$ in \eqref{RLy}, we have
\beqs
|g(t)|=\bigo(|R_\Lambda y(t)|^{1+\alpha+\varep_2}).
\eeqs

We apply Theorem \ref{simthm} to solution $R_\Lambda y(t)$ of equation \eqref{Rygood}, for $t> T$, with a sufficiently large $T>0$, constant $a=\Lambda H(v_*) $ and $f=g$  in \eqref{basicf}. It results in  the existence of a nonzero vector $\xi_*\in\R^n$ such that 
\beq\label{Rxi3}
|R_\Lambda y(t)-\xi_* t^{-1/\alpha}|=\bigo(t^{-1/\alpha-\varep_3})\text{ for some $\varep_3>0$,}
\eeq
and 
\beq\label{Hv1}
\alpha \Lambda H(v_*)  |\xi_*|^\alpha=1.
\eeq
The desired statement  \eqref{Rxi} immediately follows \eqref{Rxi3}.

Because $\xi_*=\lim_{t\to\infty} t^{1/\alpha}R_\Lambda y(t)$, by \eqref{Rxi}, and the fact $\xi_*\ne 0$, we have $\xi_*\in R_\Lambda(\R^n)\setminus\{0\}$. Hence, $\xi_*$ is an eigenvector of $A$ associated with $\Lambda$.

Next, we prove \eqref{ysim}. Writing $y(t)-\xi_* t^{-1/\alpha}=({\rm Id}-R_\Lambda)y(t)+(R_\Lambda y(t)-\xi_* t^{-1/\alpha})$, and using  the estimate  \eqref{remy}, with $\varep=\varep_0$, and estimate \eqref{Rxi3} yield
\beqs
|y(t)-\xi_* t^{-1/\alpha}| = \bigo(t^{-1/\alpha-\varep_0}+t^{-1/\alpha-\varep_3}).
\eeqs
This implies \eqref{ysim} with $\varep=\min\{\varep_0,\varep_3\}$.

Finally, we prove \eqref{xiHA}. Let $w(t)=t^{1/\alpha}y(t)$. 
As $t\to\infty$, we have $v(t)\to v_*$ and $w(t)\to \xi_*$, thanks to \eqref{vlim} and \eqref{ysim}.
By  writing $v(t)=w(t)/|w(t)|$ and passing $t\to\infty$, we obtain
\beq\label{vxi} 
v_*=\xi_*/|\xi_*|
\eeq 
Then \eqref{xiHA} follows \eqref{Hv1} and \eqref{vxi}.
The proof is complete.
\end{proof}

\section{The general case}\label{gencase}

In this section, we again study equation \eqref{yReq}, where the matrix $A$ is as in Section \ref{aesec}, and the function $G(t,x)$ is as in Assumption \ref{newG}.
Let $y(t)$ be a solution as in Section \ref{symcase}.

For $1\le k,\ell\le n$, let $E_{k\ell}$ be the elementary $n\times n$ matrix $(\delta_{ki}\delta_{\ell j})_{1\le i,j\le n}$, where $\delta_{ki}$ and $\delta_{\ell j}$ are the Kronecker delta symbols.
For $j=1,2,\ldots,d$, define
 \beq\label{RR}
\hat R_{\lambda_j}=\sum_{1\le i\le n,\Lambda_i=\lambda_j}E_{ii}\text{ and } R_{\lambda_j}=S^{-1}\hat R_{\lambda_j} S.
 \eeq
 
Then one immediately has
\beq \label{Pc} I_n = \sum_{j=1}^{d} R_{\lambda_j},
\quad R_{\lambda_i}R_{\lambda_j}=\delta_{ij}R_{\lambda_j},
\quad  AR_{\lambda_j}=R_{\lambda_j} A=\lambda_j R_{\lambda_j}.
\eeq 

Thanks to \eqref{Pc}, each $R_{\lambda_j}$ is a projection, and $R_{\lambda_j}(\R^n)$ is the eigenspace of $A$ associated with the eigenvalue $\lambda_j$.

When $A$ is symmetric,  the $R_{\lambda_j}$'s defined in \eqref{RR} are the orthogonal projections defined in Section \ref{symcase}.

Regarding the function $H$, we examine Assumption \ref{HHsphere} further.

\begin{lemma}\label{HCequiv}
Let $F$ be a function in $\mathcal H_\alpha(\R^n,\R)$ for some $\alpha>0$.
Then $F$ has property (HC) on the unit sphere if and only if $F$ has property (HC) on $\R^n\setminus\{0\}$.
\end{lemma}
\begin{proof}
It is clear that  if $F$ has property (HC) on $\R^n\setminus\{0\}$ then it has property (HC) on the unit sphere.  
 
 Now, suppose $F$ has property (HC) on the unit sphere. 
 Then $F$ is continuous on the unit sphere. Consequently, 
 \beq\label{Hsup}
\max_{|x|=1}|F(x)|=M_0\in[0,\infty).
\eeq
 
 Let $\xi\ne 0$. In Definition \ref{HCdef} for the set $E$ being the unit sphere, we take $x_0$ to be the unit vector $\xi/|\xi$. Then there are $r_1,C_0,\gamma>0$ such that if $x\in\R^n\setminus\{0\}$ and $\big|x/|x|-\xi/|\xi|\big|<r_1$, then 
 \beq\label{Hxx}
 |F(x/|x|)-F(\xi/|\xi|)|\le C_0\big|x/|x|-\xi/|\xi|\big|^\gamma.
 \eeq
 
 Note that the functions  $x\in\R^n\setminus\{0\}\mapsto |x|^\alpha$ and $x\in\R^n\setminus\{0\}\mapsto x/|x|$ are $C^1$-functions. Let $r_2=|\xi|/2$. Then there is a constant $M_1>0$ such that if $|x-\xi|<r_2$ then
 \beq\label{xxineq}
\big||x|^\alpha-|\xi|^\alpha\big|\le M_1|x-\xi|\text{ and } \big|x/|x|-\xi/|\xi|\big|\le M_1|x-\xi|.
 \eeq
 
Set $r_0=\min\{1,r_1/M_1,r_2\}$.
Let $x\in\R^n$ and $|x-\xi|<r_0$. We have $x\ne 0$ and, from the second inequality in \eqref{xxineq},
\beq \label{xr1}
\big|x/|x|-\xi/|\xi|\big|\le M_1 |x-\xi|<r_1.
\eeq 

Combing \eqref{Hxx}, \eqref{xxineq}, \eqref{xr1}  and \eqref{Hsup} yields
 \begin{align*}
| F(x)-F(\xi)|
&=\big||x|^\alpha F(x/|x|)-|\xi|^\alpha F(\xi/|\xi|)\big|\\
&\le \big||x|^\alpha -|\xi|^\alpha\big|\cdot | F(x/|x|)| + |\xi|^\alpha |F(x/|x|)- F(\xi/|\xi|)|\\
&\le M_1|x-\xi| M_0+ |\xi|^\alpha C_0\big|x/|x|-\xi/|\xi|\big|^\gamma
\le M_1M_0|x-\xi|+ |\xi|^\alpha C_0M_1^\gamma |x-\xi|^\gamma.
 \end{align*}
 Thus,
 \beqs
 | F(x)-F(\xi)|\le (M_1M_0+|\xi|^\alpha C_0M_1^\gamma)|x-\xi|^{\min\{1,\gamma\}}.
 \eeqs
Therefore, $F$ has property (HC) on $\R^n\setminus\{0\}$.
\end{proof}

Thanks to Lemma \ref{HCequiv}, Assumptions \ref{Hf} and \ref{HHsphere} can be combined into a simple form as the following.

\begin{assumption}\label{combine}
The function $H$ belongs to $\mathcal H_\alpha(\R^n,\R)$ for some $\alpha>0$, has property (HC) on the unit sphere, and $H> 0$ on the unit sphere.
\end{assumption}

Indeed, by the virtue of Lemma \ref{HCequiv}, such a function $H$ in Assumption \ref{combine} has property (HC) on $\R^n\setminus \{0\}$. Consequently, it is continuous on $\R^n\setminus \{0\}$. By the fact \eqref{Hsup} for $H$ in place of $F$, one has $|H(x)|\le M_0|x|^\alpha$ for all $x\ne 0$. Together with the fact $H(0)=0$, we have $H$ is also continuous at the origin, and, hence, on $\R^n$. Finally, for $x\ne 0$, $H(x)=|x|^\alpha H(x/|x|)> 0$. Thus, $H$ satisfies both Assumptions \ref{Hf} and \ref{HHsphere}.

We are ready to present our main result.

\begin{theorem}\label{mainthm}
Under Assumption \ref{combine}, there exist an eigenvalue $\Lambda$ of $A$ and  an eigenvector $\xi_*$ of $A$ associated with $\Lambda$ such that, as $t\to\infty$,
\beq\label{mainest}
|y(t)-\xi_* t^{-1/\alpha}|=\bigo(t^{-1/\alpha-\varep})\text{ for some $\varep>0$.}
\eeq

More specifically, one has, as $t\to\infty$,
\beq\label{newRy}
|(I_n-R_\Lambda)y(t)|,| R_\Lambda y(t)-\xi_* t^{-1/\alpha}|=\bigo(t^{-1/\alpha-\varep})\text{ for some $\varep>0$,}
\eeq
and relation  \eqref{xiHA}  holds true.
\end{theorem}
\begin{proof}
Setting $z(t)=Sy(t)$, we have $z(t)$ satisfies equation \eqref{zeq} with $\widetilde H$ and $\widetilde G$ defined in \eqref{HRz}. We will apply the results in Section \ref{symcase} to equation \eqref{zeq}.

\medskip
\textit{Verification of Assumption \ref{combine} for $\widetilde H$.} Clearly, $\widetilde H\in \mathcal H_\alpha(\R^n)$ and $\widetilde H(z)>0$ for $|z|=1$. 

Let $\xi$ be a unit vector in $\R^n$. By Lemma \ref{HCequiv}, the function $H$ has property (HC) on $\R^n\setminus\{0\}$. Let  $x_0=S^{-1}\xi\ne 0$ in Definition \ref{HCdef} for $E=\R^n\setminus\{0\}$. Then there are numbers $r\in(0,|x_0|)$ and $C,\gamma>0$ such that
\beq\label{HHS}
|H(x)-H(S^{-1}\xi)|\le C|x-S^{-1}\xi|^\gamma,
\eeq
for any $x\in \R^n$ with $|x-S^{-1}\xi|<r$. (Note that such $x$ is already a nonzero vector.) 

Take $r'=r/\|S^{-1}\|>0$. Let $z$ be any unit vector in $\R^n$ and $|z-\xi|<r'$. Set $x=S^{-1}z$. 
Then 
\beq\label{xS}
|x-S^{-1}\xi|=|S^{-1}(z-\xi)|\le \|S^{-1}\|\cdot|z-\xi|<\|S^{-1}\| r'=r.
\eeq

It follows \eqref{HHS} and  \eqref{xS} that
\beqs
|\widetilde H(z)-\widetilde H(\xi)|\le C|x-S^{-1}\xi|^\gamma\le C\|S^{-1}\|^\gamma |z-\xi|^\gamma.
\eeqs

Therefore, the function $\widetilde H$ has property (HC) on the unit sphere.

\medskip
Thanks to \eqref{Syz}, one can also verify that the function $\widetilde G$ satisfies Assumption \ref{newG}, with the same $\alpha,\delta,t_*,T_*$, and a new constant $c_*$.

\medskip
We apply Theorem \ref{thmsym} and Corollary \ref{cor1} to equation \eqref{zeq} and solution $z(t)$, with $A_0$ replacing $A$ and $\hat R_{\lambda_j}$  replacing $R_{\lambda_j}$.
Then there exist an eigenvalue $\Lambda$ of $A_0$ and an eigenvector $\xi_0$ of $A_0$ associated  with $\Lambda$ such that 
\beq\label{zxi}
|z(t)-\xi_0 t^{-1/\alpha}|=\bigo(t^{-1/\alpha-\varep}),
\eeq 
\beq\label{hatRz}
|(I_n-\hat R_\Lambda)z(t)|,|\hat R_\Lambda z(t)-\xi_0 t^{-1/\alpha}|=\bigo(t^{-1/\alpha-\varep}),
\eeq
for some number $\varep>0$, and
\beq\label{A0xi} 
\alpha \Lambda \widetilde H(\xi_0)=1.
\eeq

Let $\xi_*=S^{-1}\xi_0$ which is a nonzero vector. 
It is clear that $\Lambda$ is an eigenvalue of $A$ and  $\xi_*$ is an eigenvector of $A$ associated  with $\Lambda$. 

It follows \eqref{zxi} that 
$|S(y(t)-\xi_* t^{-1/\alpha})|=\bigo(t^{-1/\alpha-\varep})$, 
which implies \eqref{mainest}.

Similarly, we convert \eqref{hatRz} to 
\beqs
|S(I_n-R_\Lambda)y(t)|,|S( R_\Lambda y(t)-\xi_* t^{-1/\alpha})|=\bigo(t^{-1/\alpha-\varep}),
\eeqs
which yields \eqref{newRy}.

Finally, relation \eqref{A0xi} and the fact $\widetilde H(\xi_0)=H(\xi_*)$  imply \eqref{xiHA}.
The proof is complete.
\end{proof}

The following remarks on Theorem \ref{mainthm} are in order.
\begin{enumerate}[label=\rnum]
\item As a consequence of \eqref{xiHA} and \eqref{Hyal}, one can estimate $|\xi_*|$ from above and below by
\beq\label{xibound}
 \frac1{(\alpha c_2 \Lambda_n)^{1/\alpha}} \le |\xi_*|\le \frac1{(\alpha c_1 \Lambda_1)^{1/\alpha}}.
\eeq
These estimates agree with \eqref{yinfsup} derived previously for the symmetric matrix case.

\item One observes that the bounds in \eqref{xibound} are independent of the solution $y(t)$. This is different from the case of the previously studied ODE systems with the lowest order term being linear such as \eqref{uAF}.

\item The actual value of $|\xi_*|$ may still depend on the solution $y(t)$, while, for the basic case in Section \ref{simeg}, it does not.
\end{enumerate}

In the above proof of Theorem \ref{mainthm}, the properties of $\xi_*$ are obtained by using the approximate equation \eqref{Rygood} and the result proved for it in Theorem \ref{simthm}. However, as we will see below, the asymptotic approximation \eqref{mainest}, once established, already determines those properties of $\xi_*$. It is proved independently without using equation \eqref{Rygood}.  

\begin{theorem}\label{xirel}
Under Assumption \ref{combine}, let $y\in C^1([T,\infty),\R^n)$ be a  solution of \eqref{yReq} on $(T,\infty)$ for some $T\ge t_*$. 

Suppose there exist a number $p>0$ and a vector $\xi\in\R^n\setminus\{0\}$ such that
\beq\label{xipa}
|y(t)-\xi t^{-p}|=\bigo(t^{-p-\varep})\text{ for some $\varep>0$.}
\eeq

Then $p=1/\alpha$ and $\xi$ satisfies 
\beq\label{heu}
 H(\xi)A\xi =\frac1\alpha \xi.
\eeq
Consequently,  $\xi$ is an eigenvector of $A$ associated with the eigenvalue $\Lambda\eqdef 1/(\alpha H(\xi))$. 
\end{theorem}
\begin{proof}
On the one hand, having \eqref{Adiag}, we set $z(t)=Sy(t)$ and obtain equation \eqref{zeq}. Then we can establish estimates in \eqref{ypower} for $z(t)$, and, hence, thanks to \eqref{Syz}, for $y(t)$ itself. Consequently, there exist $T>0$ and $c\ge 1$ such that
\beq\label{yalpha}
\frac{t^{-1/\alpha}}{c} \le |y(t)|\le c t^{-1/\alpha}\text{ for all $t\ge T$.}
\eeq

On the other hand, thanks to \eqref{xipa}, there exist $T'>0$ and $c'\ge 1$ such that
\beq\label{yp}
\frac{ t^{-p}}{c'} \le |y(t)|\le c' t^{-p}\text{ for all $t\ge T'$.}
\eeq

Comparing \eqref{yalpha} and \eqref{yp}, we must have $p=1/\alpha$. Thus,  \eqref{xipa} becomes
\beq\label{xiaprox}
|y(t)-\xi t^{-1/\alpha}|=\bigo(t^{-1/\alpha-\varep}).
\eeq

Set $w(t)=t^{1/\alpha}y(t)$. Then \eqref{xiaprox} implies 
\beq\label{limw}
|w(t)- \xi|=\bigo(t^{-\varep}) \text{ and, consequently, }
\lim_{t\to\infty} w(t)= \xi .
\eeq 

We find an asymptotic approximation for $w'(t)$. For sufficiently large $t$, we calculate
\beqs
w'
=\frac1{\alpha t} w+t^{1/\alpha}(-H(y)Ay+G(t,y))
=\frac1{\alpha t} w -\frac1t H(w)Aw+t^{1/\alpha}G(t,y).
\eeqs
Then
\beq\label{weq}
w'=\frac1t\left(\frac1\alpha \xi-H(\xi)A\xi\right) +h(t),
\eeq
where 
\beqs
h(t)=\frac1{\alpha t} (w(t)-\xi) -\frac1t (H(w(t))Aw(t)-H(\xi)A\xi)+t^{1/\alpha}G(t,y(t)).
\eeqs 

We estimate $|h(t)|$. For the first term of $h(t)$, we have
\beqs
\frac1{\alpha t} |w(t)-\xi|=\bigo(t^{-1-\varep}).
\eeqs

For the middle term  of $h(t)$, we have
\beqs
H(w)Aw-H(\xi)A\xi=(H(w)-H(\xi))Aw+H(\xi)A(w-\xi).
\eeqs

By Lemma \ref{HCequiv}, the function $H$ has property (HC) on the set $E=\R^n\setminus\{0\}$.
By Definition \ref{HCdef}  with $F=H$ and $x_0=\xi$,  and the facts in \eqref{limw}, there exists $\gamma>0$ such 
\beqs
|H(w(t))-H(\xi)| \cdot |Aw(t)|=\bigo(|w(t)-\xi|^\gamma\cdot 1) =\bigo(t^{-\gamma\varep}).
\eeqs

Clearly, $H(\xi)|A(w(t)-\xi)|=\bigo(|w(t)-\xi|)=\bigo(t^{-\varep})$. Thus,
\beqs
t^{-1}|H(w(t))Aw(t)-H(\xi)A\xi|=\bigo(t^{-1-\varep}+ t^{-1-\gamma\varep}).
\eeqs

For the last term of $h(t)$, it follows \eqref{Gcond} and \eqref{yalpha} that
\beqs
t^{1/\alpha}|G(t,y(t))| =\bigo(t^{1/\alpha}\cdot t^{-(1+\alpha+\delta)/\alpha})=\bigo(t^{-1-\delta/\alpha})  .
\eeqs

Therefore,  
\beq\label{hsec6}
|h(t)|=\bigo(t^{-1-\varep}+ t^{-1-\gamma\varep}+t^{-1-\delta/\alpha})=\bigo(t^{-1-\delta_1}),
\eeq 
where $\delta_1=\min\{\varep,\gamma\varep,\delta/\alpha\}$.

Let $t>s$ be sufficiently large numbers. Integrating equation \eqref{weq} from $s$ to $t$, and taking into account \eqref{hsec6} give
\beqs
\left| w(t)-w(s)-\left (\frac1\alpha \xi-H(\xi)A\xi\right)\ln (t/s)\right|\le C(s^{-\delta_1}-t^{-\delta_1}),
\eeqs
for some constant $C>0$.
Letting $t=2s$ and taking $s\to\infty$ yield
\beqs
\frac1\alpha \xi-H(\xi)A\xi=0,
\eeqs
which proves \eqref{heu}. The last statement in Theorem \ref{xirel} is an obvious consequence of \eqref{heu}.
\end{proof}

Below, we provide explicit examples for $H$ satisfying Assumption \ref{combine}.
Firstly, here are some elementary facts.

\begin{lemma}\label{HCalg}
Let $E$ be a nonempty subset of $\R^n$. Suppose two functions $F_1,F_2:E\to \R$ have property (HC) on E. Then so do the functions $aF_1+bF_2$ and $F_1F_2$ for any $a,b\in \R$.
\end{lemma}
\begin{proof}
Suppose $F_i$, for $i=1,2$, satisfies \eqref{FHder} with $r_i,C_i,\gamma_i>0$.  
Let $r=\min\{1,r_1,r_2\}$, $C=\max\{C_1,C_2\}$ and $\gamma=\min\{\gamma_1,\gamma_2\}$. Then both $F_1$ and $F_2$ satisfy \eqref{FHder} with the same numbers $r,C,\gamma$. The statement for $aF_1+bF_2$ is now obviously true. For the product $F_1F_2$, one observes that $F_1$ and $F_2$ are bounded on the set $\{x\in E: |x-x_0|<r\}$. The rest of the proof is standard. We omit the details.
\end{proof}

For $p\ge 1$,  the $\ell^p$-norm of $x=(x_1,\ldots,x_n)\in\R^n$ is $\|x\|_p=(\sum_{j=1}^n |x_j|^p)^{1/p}$.

\begin{lemma}\label{Hpower}
Let $E=\R^n\setminus\{0\}$ and $p\ge 1$, $\alpha>0$.
Then the function $H(x)=\|x\|_p^\alpha$, for $x\in\R^n$,  belongs to $\mathcal H_\alpha(\R^n,\R)$ and has property (HC) on E with the same power $\gamma=1$ in \eqref{FHder}.
\end{lemma}
\begin{proof}
The fact that $H$ belongs to $\mathcal H_\alpha(\R^n,\R)$ is obvious.
When $p>1$, one has $H\in C^1(E)$, hence, $H$ has property (HC) on E with the same power $\gamma=1$.

Consider the case $p=1$.
Let $\xi\in E$ and set $r=\|\xi\|_1>0$. 
Note that there is $C>0$ such that  
$$|t^\alpha-r^\alpha|\le C|t-r|\text{ for all } t\in I\eqdef [r/2,3r/2].$$

Let $x\in E$ with  $|x-\xi|<r/(2\sqrt n)$.
One has 
$$\big | \|x\|_1-\|\xi\|_1\big| 
\le \|x-\xi\|_1\le \sqrt{n}|x-\xi|<r/2,$$
which implies $t\eqdef \|x\|_1\in I$. Thus,
\beqs
|H(x)-H(\xi)|=\big | \|x\|_1^\alpha-\|\xi\|_1^\alpha\big| 
\le C \big | \|x\|_1-\|\xi\|_1\big| 
\le C\sqrt{n}|x-\xi|,
\eeqs
which proves  that $H$ has property (HC) on E with the same power $\gamma=1$.
\end{proof}

\begin{example}\label{Heg} The requirement of $H$ having property (HC) on the unit sphere is not strict. In many cases, $H$, in fact, is a $C^1$-function on $\R^n\setminus\{0\}$, hence, it meets this condition. For example, as in Section \ref{simeg}, $H(x)=|x|^\alpha$ for $\alpha>0$. A generalization is 
\beq\label{Hprod}
H(x)=\|K_1 x\|_{p_1}^{\alpha_1}\|K_2 x\|_{p_2}^{\alpha_2}\ldots \|K_m x\|_{p_m}^{\alpha_m},
\eeq
where, for $j=1,2,\ldots, m$,  $K_j$ is an invertible $n\times n$ matrix, $\alpha_j>0$, and $p_j>1$. 
In this case, $H$ satisfies Assumption \ref{combine} with 
\beq\label{alph} 
\alpha=\alpha_1+\alpha_2+\ldots+\alpha_m.
\eeq
In fact, thanks to Lemmas \ref{Hpower} and \ref{HCalg}, we allow $p_j\ge 1$ in \eqref{Hprod}.

Of course, $H$ can also be a linear combination, with positive coefficients, of the functions of the form \eqref{Hprod} having possibly  different $m$'s, but resulting in the same $\alpha$ in \eqref{alph}. 

The function $H$ can be even more complicated. Here are some examples when $n=2$.
For $x=(x_1,x_2)\in\R^2,$
\begin{align*}
&H(x)=(x_1^2+3x_2^2)^{3/4},&&
H(x)=(x_1^4-x_1^2 x_2^2 +x_2^4)^{5/3},\\
&H(x)=(|x_1^6+5x_2^6|^{4/3}+|2x_1^6-x_2^6|^{4/3})^{1/7},&&
H(x)=(\|x\|_{5/3}^6+\|x\|_{7/4}^6)^{11/8},\\
&H(x)=\sqrt{|x_1|}+\sqrt{|x_2|}.
\end{align*}

Note that the last function $H$ belongs to $\mathcal H_{1/2}(\R^2,\R)$ and has property (HC) on $\R^2\setminus\{0\}$ with the same power $\gamma=1/2$, but is not a $C^1$-function on $\R^2\setminus\{0\}$.

More examples of $H$ can be constructed by the similar investigation in section 6 of \cite{CaHK1}.
\end{example}

\begin{remark}\label{other}
The study of the solutions of ODE systems near an equilibrium has a long history.
One of the long-standing methods for detailed descriptions of their asymptotic behavior is the Poincar\'e--Dulac normal form \cite{ArnoldODEGeo,BibikovBook,LefschetzBook}.
This has been developed much further by many researchers over the years, see the books \cite{BrunoBook1989,BrunoBook2000}, monograph  \cite{KFbook2013}, and, for example, papers \cite{Bruno2004,Bruno2008c,Bruno2012,Bruno2018} and references therein. 
However, the techniques from this approach, such as the generalized normal forms and power geometry in \cite{BrunoBook1989,BrunoBook2000}, are not applicable to the equations of our current interest. In fact, our class of equations, problems, techniques and those in \cite{BrunoBook1989,BrunoBook2000,KFbook2013} are quite different, and can be considered as complementary to each other. For instance, the main task in \cite{Bruno2004,Bruno2008c,Bruno2012,Bruno2018,KFbook2013} is to \textit{find a solution} with a certain type of expansions. 
On contrary, our result establishes the exact asymptotic approximation for \textit{any} (nontrivial, decaying) solution.
\end{remark}



\begin{thebibliography}{10}

\bibitem{ArnoldODEGeo}
V.~I. Arnol{\cprime}d.
\newblock {\em {Geometrical methods in the theory of ordinary differential
  equations}}, volume 250 of {\em {A Series of Comprehensive Studies in
  Mathematics}}.
\newblock Springer-Verlag, New York, second edition, 1988.

\bibitem{BibikovBook}
Yuri~N. Bibikov.
\newblock {\em Local theory of nonlinear analytic ordinary differential
  equations}, volume 702 of {\em Lecture Notes in Mathematics}.
\newblock Springer-Verlag, Berlin-New York, 1979.

\bibitem{Bruno2004}
A.~D. Bruno.
\newblock Asymptotic behavior and expansions of solutions of an ordinary
  differential equation.
\newblock {\em Uspekhi Mat. Nauk}, 59(3(357)):31--80, 2004.

\bibitem{Bruno2008c}
A.~D. Bruno.
\newblock Power-logarithmic expansions of solutions of a system of ordinary
  differential equations.
\newblock {\em Dokl. Akad. Nauk}, 419(3):298--302, 2008.

\bibitem{Bruno2012}
A.~D. Bruno.
\newblock Power-exponential expansions of solutions of an ordinary differential
  equation.
\newblock {\em Dokl. Akad. Nauk}, 444(2):137--142, 2012.

\bibitem{Bruno2018}
A.~D. Bruno.
\newblock On complicated expansions of solutions to {ODES}.
\newblock {\em Comput. Math. Math. Phys.}, 58(3):328--347, 2018.

\bibitem{BrunoBook1989}
Alexander~D. Bruno.
\newblock {\em Local methods in nonlinear differential equations}.
\newblock Springer Series in Soviet Mathematics. Springer-Verlag, Berlin, 1989.

\bibitem{BrunoBook2000}
Alexander~D. Bruno.
\newblock {\em Power geometry in algebraic and differential equations},
  volume~57 of {\em North-Holland Mathematical Library}.
\newblock North-Holland Publishing Co., Amsterdam, 2000.

\bibitem{CaH2}
Dat Cao and Luan Hoang.
\newblock Asymptotic expansions in a general system of decaying functions for
  solutions of the {N}avier-{S}tokes equations.
\newblock {\em Ann. Mat. Pura Appl. (4)}, 199(3):1023--1072, 2020.

\bibitem{CaH1}
Dat Cao and Luan Hoang.
\newblock Long-time asymptotic expansions for {N}avier-{S}tokes equations with
  power-decaying forces.
\newblock {\em Proc. Roy. Soc. Edinburgh Sect. A}, 150(2):569--606, 2020.

\bibitem{CaH3}
Dat Cao and Luan Hoang.
\newblock Asymptotic expansions with exponential, power, and logarithmic
  functions for non-autonomous nonlinear differential equations.
\newblock {\em J. Evol. Equ.}, 21(2):1179--1225, 2021.

\bibitem{CaHK1}
Dat Cao, Luan Hoang, and Thinh Kieu.
\newblock Infinite series asymptotic expansions for decaying solutions of
  dissipative differential equations with non-smooth nonlinearity.
\newblock {\em Qual. Theory Dyn. Syst.}, 20(3):Paper No. 62, 38 pp, 2021.

\bibitem{FS84a}
C.~Foias and J.-C. Saut.
\newblock {Asymptotic behavior, as {$t\rightarrow +\infty $}, of solutions of
  {N}avier-{S}tokes equations and nonlinear spectral manifolds}.
\newblock {\em Indiana Univ. Math. J.}, 33(3):459--477, 1984.

\bibitem{FS87}
C.~Foias and J.-C. Saut.
\newblock {Linearization and normal form of the {N}avier-{S}tokes equations
  with potential forces}.
\newblock {\em Ann. Inst. H. Poincar{\'e} Anal. Non Lin{\'e}aire}, 4(1):1--47,
  1987.

\bibitem{Ghidaglia1986a}
Jean-Michel Ghidaglia.
\newblock Long time behaviour of solutions of abstract inequalities:
  applications to thermohydraulic and magnetohydrodynamic equations.
\newblock {\em J. Differential Equations}, 61(2):268--294, 1986.

\bibitem{H5}
Luan Hoang.
\newblock Asymptotic expansions about infinity for solutions of nonlinear
  differential equations with coherently decaying forcing functions.
\newblock {\em The Annali della Scuola Normale di Pisa - Classe di Scienze},
  pages 1--48, 2022.
\newblock accepted. Preprint DOI:10.48550/arXiv.2108.03724.

\bibitem{H6}
Luan Hoang.
\newblock The {N}avier--{S}tokes equations with body forces decaying coherently
  in time.
\newblock pages 1--36, 2022.
\newblock submitted. Preprint DOI:10.48550/arXiv.2204.05247.

\bibitem{HM1}
Luan~T. Hoang and Vincent~R. Martinez.
\newblock Asymptotic expansion in {G}evrey spaces for solutions of
  {N}avier-{S}tokes equations.
\newblock {\em Asymptot. Anal.}, 104(3--4):167--190, 2017.

\bibitem{HM2}
Luan~T. Hoang and Vincent~R. Martinez.
\newblock Asymptotic expansion for solutions of the {N}avier-{S}tokes equations
  with non-potential body forces.
\newblock {\em J. Math. Anal. Appl.}, 462(1):84--113, 2018.

\bibitem{HTi1}
Luan~T. Hoang and Edriss~S. Titi.
\newblock Asymptotic expansions in time for rotating incompressible viscous
  fluids.
\newblock {\em Ann. Inst. H. Poincar\'{e} Anal. Non Lin\'{e}aire},
  38(1):109--137, 2021.

\bibitem{KFbook2013}
Valery~V. Kozlov and Stanislav~D. Furta.
\newblock {\em Asymptotic solutions of strongly nonlinear systems of
  differential equations}.
\newblock Springer Monographs in Mathematics. Springer, Heidelberg, 2013.
\newblock Translated from the 2009 Russian second edition by Lester J.
  Senechal.

\bibitem{LefschetzBook}
Solomon Lefschetz.
\newblock {\em Differential equations: geometric theory}.
\newblock Dover Publications, Inc., New York, second edition, 1977.

\bibitem{Minea}
Gheorghe Minea.
\newblock {Investigation of the {F}oias-{S}aut normalization in the
  finite-dimensional case}.
\newblock {\em J. Dynam. Differential Equations}, 10(1):189--207, 1998.

\bibitem{Shi2000}
Y.~Shi.
\newblock A {F}oias-{S}aut type of expansion for dissipative wave equations.
\newblock {\em Comm. Partial Differential Equations}, 25(11-12):2287--2331,
  2000.

\end{thebibliography}
\def\cprime{$'$}

 \end{document}